\documentclass{article}
\usepackage[utf8]{inputenc}
\usepackage{,hyperref}
\usepackage{caption,subcaption,longtable}
\usepackage{graphics, setspace}
\usepackage{array,multirow}
\usepackage{tikz,float}
\usepackage{amsmath,amsthm,amssymb}
\usepackage{enumerate,enumitem,listings}
\usetikzlibrary{decorations.markings}
\usetikzlibrary{plotmarks}
\usepackage[T1]{fontenc}
\usetikzlibrary{patterns}

\newcommand{\fe}[2]{(x_{#1},x_{#2})}
\newcommand{\hookdoubleheadrightarrow}{%
  \hookrightarrow\mathrel{\mspace{-15mu}}\rightarrow
}
\newtheorem{theoremalph}{Theorem}

\newtheorem{coralph}[theoremalph]{Corollary}

\newtheorem{theorem}{Theorem}[section]
\newtheorem{lemma}[theorem]{Lemma}
\newtheorem{prop}[theorem]{Proposition}
\newtheorem{cor}[theorem]{Corollary}

\newtheorem*{theorem*}{Theorem}
\newtheorem*{lemma*}{Lemma}
\newtheorem*{proposition*} {Proposition}

\theoremstyle{definition}
\newtheorem{definition}[theorem]{Definition}
\newtheorem{remark}[theorem]{Remark}

\setlist[enumerate]{align=left}

\title{Link Conditions for Cubulation}
\author{Calum J. Ashcroft}
\date{\vspace{-2em}}

\begin{document}
\maketitle
	\begin{abstract}
We provide a condition on the links of polygonal complexes that is sufficient to ensure groups acting properly discontinuously and cocompactly on such complexes contain a virtually free codimension-$1$ subgroup. We provide stronger conditions on the links of polygonal complexes, which are sufficient to ensure groups acting properly discontinuously and cocompactly on such complexes act properly discontinuously on a $CAT(0)$ cube complex. If the group is hyperbolic then this action is also cocompact, hence by Agol's Theorem the group is virtually special (in the sense of Haglund--Wise); in particular it is linear over $\mathbb{Z}$. We consider some applications of this work. Firstly, we consider the groups classified by \cite{Kangaslampi-Vdovina} and \cite{Carbone-Kangaslampli-Vdovina_2012}, which act simply transitively on $CAT(0)$ triangular complexes with the minimal generalized quadrangle as their links, proving that these groups are virtually special. We further apply this theorem by considering generalized triangle groups, in particular a subset of those considered by \cite{Caprace-Conder-Kaluba-Witzel_triangle}. 
	\end{abstract}
%----------------------------------------------------------------------------------------
%   INTRODUCTION
%----------------------------------------------------------------------------------------	

	\section{Introduction}
%----------------------------------------------------------------------------------------
%  Cubulating groups acting on polygonal complexes
%----------------------------------------------------------------------------------------	

	\subsection{Cubulating groups acting on polygonal complexes}
Recently, a very fruitful route to understanding groups has been to find an action on a $CAT(0)$ cube complex. Indeed, an action without a global fixed point provides an obstruction to Property $(T)$ \cite{Niblo-Reeves97}, while a proper action is enough to guarantee the Haagerup property \cite{Cheriz-Martin-Valette_haagerupproperty}. Further properties, such as residual finiteness or linearity, can deduced if the cube complex is \emph{special} \cite{Haglund-Wise}. Perhaps the most notable recent use of cube complexes was in Agol's proof of the Virtual Haken Conjecture \cite{Agol13}. 

Therefore, it is of interest to find actions of groups on $CAT(0)$ cube complexes. In this paper we provide a condition on the links of polygonal complexes (including those with triangular faces) that is sufficient to ensure a group acting properly discontinuously and cocompactly on such a complex contains a virtually free codimension-$1$ subgroup. We provide stronger conditions that are sufficient to ensure a group acting properly discontinuously and cocompactly on such a complex acts properly discontinuously on a $CAT(0)$ cube complex: in many applications (in particular for hyperbolic groups) this action is also cocompact. We shall see that these conditions can be practically checked in many examples, and can in fact be checked by computer search if desired.

For a polygonal complex $X$ and a vertex $v$ we define the \emph{link} of $v$, $Lk_{X}(v)$ (or simply $Lk(v)$ when $X$ is clear from context), as the graph whose vertices are the edges of $X$ incident at $v$, and two vertices $e_{1}$ and $e_{2}$ are connected by an edge $f$ in $Lk(v)$ if the edges $e_{1}$ and $e_{2}$ in $X$ are adjacent to a common face $f$. We can endow the link graph with the \emph{angular metric}: an edge $f=(e_{1},e_{2})$ in $Lk(v)$ has length $\alpha$, where $\alpha$ is the angle between $e_{1}$ and $e_{2}$ in the shared face $f$. We refer the reader to Section \ref{subsec: Link conditions} for further definitions, such as that of a gluably $\pi$-separated complex (this requires a solution to a system of linear equations called the \emph{gluing equations}). We note that in all of our applications, the gluing equations can be solved by considering only the links of vertices of $G\backslash X$.

	  It is well known that a group containing a codimension-$1$ subgroup cannot have Property $(T)$ \cite{Niblo-Roller1998}. Furthermore,  a hyperbolic group acting properly discontinuously and cocompactly on a $CAT(0)$ cube complex is virtually special \cite[Theorem $1.1$]{Agol13} (see Haglund-Wise \cite{Haglund-Wise} for a discussion of the notion of specialness); in particular it is linear over $\mathbb{Z}$ and is residually finite. 
	 
	\begin{theoremalph}\label{mainthm: cubulating groups}
		Let $G$ be a group acting properly discontinuously and cocompactly on a simply connected  $CAT(0)$ polygonal complex $X$.
		\begin{enumerate}[label=(\roman*)]
		    \item 
		If $G\backslash X$ is gluably weakly $\pi$-separated, then $G$ contains a virtually-free codimension-$1$ subgroup (and therefore does not have Property $(T)$). 
		
		\item If $G\backslash X$ is gluably $\pi$-separated, then $G$ acts properly discontinuously on a $CAT(0)$ cube complex. If, in addition, $G$ is hyperbolic, then this action is cocompact. In particular, if $G$ is hyperbolic, then it is virtually special, and so linear over $\mathbb{Z}$.
	\end{enumerate}
	\end{theoremalph}
	 
	It is commonly far easier to check a local property than a global one, and so local to global principles are frequently of great use. When working with complexes, it is often most natural to consider local properties related to the links of vertices. In terms of metric curvature, one of the best-known local to global principles is Gromov's Link Condition \cite[$4.2A$]{Gromov_hyperbolic}. Switching to group theoretic properties, \.{Z}uk \cite{zuk1996} and Ballmann--\'{S}wiatkowski \cite{Ballmann-Swiatkowski} independently provided a condition on the first eigenvalue of the Laplacian of links of simplicial complexes that is sufficient to prove a group acting properly discontinuously and cocompactly on such a complex has Property $(T)$.
	
In particular, as opposed to \cite[Example $4.3$]{Hruska-Wise}, we do not require a partition of the edges of links into cut sets: we can remove this assumption at the expense of requiring that every cutset contains at least two elements, and that the \emph{gluing equations} are satisfied for the cutsets (these equations are trivially satisfied for a collection of proper disjoint edge cutsets). Furthermore, we do not require that the cutsets are two-sided: $\Gamma - C$ is allowed to contain arbitrarily many components. Finally, we allow cutsets to be comprised of vertices or edges. Though we are not always able to cocompactly cubulate non-hyperbolic groups with this method, we can still produce codimension-$1$ subgroups, and often a proper action on a cube complex.

-------------------------------------------
%  Applications of the main theorem
%----------------------------------------------------------------------------------------	

\subsection{Applications of the main theorem}

We provide some applications of Theorem \ref{mainthm: cubulating groups}. 
We consider the groups classified by Kangaslampi--Vdovina \cite{Kangaslampi-Vdovina} and Carbone--Kangaslampi--Vdovina \cite{Carbone-Kangaslampli-Vdovina_2012}. These are groups that act simply transitively on triangular hyperbolic buildings: in particular, they act properly discontinuously and cocompactly on a simply connected triangular complex with links isomorphic to the minimal generalized quadrangle. There is little known about these groups: until now they were not even known to be residually finite. We apply Theorem \ref{mainthm: cubulating groups} to these groups to deduce that they are virtually special.

The full automorphism groups of Kac--Moody buildings of $2$-spherical type of large thickness have Property $(T)$ \cite{Dymara-Januszkiewicz2002,Ershov-Rall18}: neither \cite{Dymara-Januszkiewicz2002} nor \cite{Ershov-Rall18} record whether Property (T) fails at small thicknesses. Some of the groups considered in Corollary \ref{mainthm: cubulating the generalized quadrangle} are cocompact lattices in a $2$-spherical Kac--Moody building with small thickness \cite{Carbone-Kangaslampli-Vdovina_2012}. Therefore, Corollary \ref{mainthm: cubulating the generalized quadrangle} complements \cite{Dymara-Januszkiewicz2002,Ershov-Rall18}, providing an example example of the failure of Property $(T)$ when the thickness is small.

	\begin{coralph}\label{mainthm: cubulating the generalized quadrangle}
		Let $X$ be a simply connected polygonal complex such that every face has at least $3$ sides, and the link of every vertex is isomorphic to the minimal generalized quadrangle. If a group $G$ acts properly discontinuously and cocompactly on $X$, then it is virtually special; in particular it is linear over $\mathbb{Z}$. 
	\end{coralph}
	We prove that if $X$ and $G$ are as above, then $X$ can be endowed with a $CAT(0)$ metric such that $G\backslash X$ is gluably $\pi$-separated. However, we show that that it is not disjointly $\pi$-separated, so that \cite[Example 4.3]{Hruska-Wise} cannot be applied to such a complex.

	As a further application of Theorem \ref{mainthm: cubulating groups}, we consider generalized triangle groups, as defined in \cite{Lubotzky-Manning-Wilton} (see Definitions \ref{def: generalized triangle} and \ref{def: generalized triangle 2}). Let $C_{k,2}$ be the cage graph on $k$ edges, i.e. the smallest $k$ regular graph of girth $2$. For finite-sheeted covering graphs $\Gamma_{i} \looparrowright C_{k,2}$, we consider an associated pair of families of triangular complexes of groups $D^{j}_{0,k}(\Gamma_{1},\Gamma_{2},\Gamma_{3})$, and $D^{j}_{k}(\Gamma_{1},\Gamma_{2},\Gamma_{3})$. We remark that these complexes of groups are not necessarily unique for given $\Gamma_{1},\Gamma_{2},\Gamma_{3}$.

We consider explicitly the graphs used in \cite{Caprace-Conder-Kaluba-Witzel_triangle}: we refer to them by their Foster Census names (see \cite{fostercensus}). The only graph not in the Foster Census is $G54$, the Gray graph, which is edge but not vertex transitive. Using Theorem \ref{mainthm: cubulating groups} and Theorem \ref{mainthm: cubulating generalized triangle groups} we can deduce the following.
	
		\begin{coralph}\label{coralph: small girth generalized triangle groups}
	Let $\Gamma_{i}\looparrowright C_{k,2}$ be finite-sheeted covers, such that $girth(\Gamma_{i})\geq 6$ for each $i$. Let $G=\pi_{1}(D^{j}_{0,k}(\Gamma_{1},\Gamma_{2},\Gamma_{3}))$ or $G=\pi_{1}(D^{j}_{k}(\Gamma_{1},\Gamma_{2},\Gamma_{3}))$ for some $j$.
	\begin{enumerate}[label=(\roman*)]
	    \item If $\Gamma_{i}\in\{F24A,\;F26A,\;F48A\}$ for each $i$, then $G$ acts properly discontinuously on a $CAT(0)$ cube complex: if $G$ is hyperbolic, then this action is also cocompact and so $G$ is virtually special.
	     \item If $\Gamma_{1}\in\{F40A,G54\}$, then $G$ acts properly discontinuously on a $CAT(0)$ cube complex: if $G$ is hyperbolic, then this action is also cocompact and so $G$ is virtually special.
	\end{enumerate}
	\end{coralph}
There are $252$ groups considered in \cite{Caprace-Conder-Kaluba-Witzel_triangle}, of which they show that $168$ do not satisfy Property $(T)$. Our method recovers this result for $101$ groups, and proves that $30$ new groups do not have Property $(T)$. We prove that each of the $131$ groups we consider has a proper action on a $CAT(0)$ cube complex, and so, by e.g. \cite{Cheriz-Martin-Valette_haagerupproperty}, has the Haagerup property. Furthermore, $125$ of these groups are hyperbolic and have a proper and cocompact action on a $CAT(0)$ cube complex, and hence by \cite{Agol13} are virtually special.

	Wise's malnormal special quotient theorem \cite{Wise-MSQT} (c.f. \cite{AgolGrovesManningMSQT}) is one of the most important theorems in modern geometric group theory. However, the proof of this theorem is famously complex and so in Section \ref{subsection: cubulating dehn fillings of generalized triangle groups} we apply Theorem \ref{mainthm: cubulating groups} to generalized triangle groups to recover partial consequences of the malnormal special quotient theorem in Corollary \ref{mainthm: cubulating dehn fillings of generalized triangle groups}. Although this theorem follows from Wise's proof of the MSQT, a far more general theorem, the proof of Corollary \ref{mainthm: cubulating dehn fillings of generalized triangle groups} is considerably shorter and simpler, and provides an effective bound on the index of the fillings required.

-------------------------------------------
%  Structure of the paper
%----------------------------------------------------------------------------------------	

\subsection{Structure of the paper}	
The main idea of the proof is the following. Since $G\backslash X$ is $\pi$-separated, we can find a collection of local geodesics in $G\backslash X$ that are locally separating at vertices of $G\backslash X$. The \emph{gluing equations} provide us with a way to glue these local geodesics together to find a locally geodesic locally separating subcomplex of $G\backslash X$: by lifting we find a geodesic separating subcomplex of $X$ with cocompact stabilizer. We then use the construction of Sageev \cite{Sageev-95}, generalized by Hruska--Wise in \cite{Hruska-Wise}, to construct the desired $CAT(0)$ cube complex. 

	The paper is structured as follows. In Section \ref{section: Cubulating hyperbolic groups acting on polygonal complexes} we define hypergraphs, which will be separating subspaces constructed in the polygonal complex, and show certain subgroups of their stabilizers are codimension-$1$. We then prove Theorem \ref{mainthm: cubulating groups} by using Hruska--Wise's \cite{Hruska-Wise} extension of Sageev's \cite{Sageev-95} construction of a $CAT(0)$ cube complex, and proving that there are `enough' hypergraphs to `separate' the polygonal complex. In Section \ref{section: finding cutsets}, we discuss how to find `separated' cutsets of a graph by computer search. In Section \ref{section: generalized quadrangles} we prove Corollary \ref{mainthm: cubulating the generalized quadrangle} by proving that the minimal generalized quadrangle is \emph{weighted edge $3$-separated} and  and endowing the polygonal complexes with a suitable $CAT(0)$ metric. In Section \ref{sec: gen triangles} we prove Theorem \ref{mainthm: cubulating generalized triangle groups} and Corollary \ref{coralph: small girth generalized triangle groups}. We again apply Theorem \ref{mainthm: cubulating groups} to prove Corollary \ref{mainthm: cubulating dehn fillings of generalized triangle groups} by considering cutsets in covers of graphs. 
	
	\section*{Acknowledgements}
I would like to thank my PhD advisor Henry Wilton for suggesting this topic, for the many useful discussions, and invaluable comments on an earlier draft of this manuscript.  I would also like to thank Pierre-Emmanuel Caprace for the extremely helpful comments and corrections on an earlier draft, as well as pointing out the relevance of Corollary \ref{mainthm: cubulating the generalized quadrangle} to Property (T) in Kac--Moody buildings.
%----------------------------------------------------------------------------------------
%   CUBULATING GROUPS
%----------------------------------------------------------------------------------------
	\section{Cubulating groups acting on polygonal complexes}\label{section: Cubulating hyperbolic groups acting on polygonal complexes}
This section is structured as follows. We first define the required conditions on graphs and complexes in Section \ref{subsec: Link conditions}, and in Section \ref{subsection: Removing cut edges and producing two sided cut sets} we discuss how to remove cut edges from links. We provide some examples where our conditions can be readily verified for graphs in Section \ref{subsection: Examples of separated graphs} and for complexes in Section \ref{subsection: Examples of solutions of the gluing equations}. We use these definitions in Sections \ref{subsection: Constructing hypergraphs in polygonal complexes}, \ref{subsection: Hypergraphs are separating}, and \ref{subsection: Hypergraph stabolizers and wallspaces} to build separating convex trees in polygonal complexes, and in Section \ref{subsection: Cubulating  groups acting on polygonal complexes} we use these convex trees, and a construction due to \cite{Sageev-95} and \cite{Hruska-Wise}, to prove Theorem \ref{mainthm: cubulating groups}.
Firstly, we introduce the relevant definitions for links.
\subsection{Some separation conditions}\label{subsec: Link conditions}

We now define the notion of `separatedness' of a graph. The \emph{combinatorial metric} on a graph $\Gamma$ is the path metric induced by assigning each edge of $\Gamma$ length $1$.
\begin{definition}
	Let $\Gamma$ be a finite metric graph.
	\begin{enumerate}[label=\roman*)]
	    \item 

	An edge $e$ is a \emph{cut edge} if $\Gamma-\{e\}$ is disconnected.
		    \item A set $C\subseteq \Gamma$ is a \emph{cutset} if $\Gamma - C$ is disconnected as a topological space.
	    \item 	A cutset $C$ is an \emph{edge cutset} if $C\subseteq E(\Gamma)$ and is a \emph{vertex cutset} if $C\subseteq V(\Gamma).$
	    \item An edge cutset $C$ is \emph{proper} if for any edge $e\in C$, the endpoints of $e$ lie in disjoint components of $\Gamma-C$.
	    \item A vertex cutset $C$ is \emph{proper} if for any vertex $u\in C$, and any distinct vertices $v,w$ adjacent to $u$, the vertices $v$ and $w$ lie in disjoint components of $\Gamma-C$.
	   \end{enumerate}
For an edge $e$ in $\Gamma$ let $m(e)$ be the midpoint of $e$. For $\sigma >0$ a set $\mathcal{C}\subseteq E(\Gamma)$ is \emph{$\sigma$-separated} if for all distinct $e_{1},e_{2}\in\mathcal{C}$, $d_{\Gamma}(m(e_{1}),m(e_{2}))\geq \sigma .$ A set $\mathcal{C}\subseteq V(\Gamma)$ is \emph{$\sigma$-separated} if for all distinct $v_{1},v_{2}\in\mathcal{C}$, $d_{\Gamma}(v_{1},v_{2})\geq \sigma .$
\end{definition}
\begin{remark}
We note that proper cut sets are very natural to consider. Any minimal edge cut set is proper, and more importantly, proper cutsets are preserved under passing to finite covers.

Finding proper edge cutsets is easy, but for a given graph $\Gamma$ there may not be any proper $\sigma$-separated vertex cutsets: see for example the graph $F26A$, considered in Lemma \ref{lem: F26A is * separated}.
\end{remark}
\begin{definition}[\emph{Edge separated}]
	    
Let $\Gamma$ be a finite metric graph, and let $\sigma>0$. We will say that $\Gamma$ is \emph{edge $\sigma$-separated} if $\Gamma$ is connected, contains no vertices of degree $1$, and there exists a collection of proper $\sigma$-separated edge cutsets $C_{i}\subseteq E(\Gamma)$ with $\cup_{i}C_{i}=E(\Gamma)$ and $\vert C_{i}\vert\geq 2$ for each $i$.

	We say the graph is \emph{disjointly edge $\sigma$-separated} if the above cutsets form a partition of the edges.
\end{definition}
Note that to each edge cutset $C$ we can assign a partition $\mathcal{P}(C)$ to $\pi_{0}(\Gamma-C)$: we \emph{always require} that such a partition is at least as coarse as connectivity in $\Gamma - C$, and each partition contains at least two elements. The \emph{canonical partition} of $C$ is that induced by connectivity in $\Gamma -C$.

\begin{definition}[\emph{Strongly edge separated}]
 	A graph $\Gamma$ is \emph{strongly edge $\sigma$-separated} if $\Gamma$ is edge $\sigma$-separated and for every pair of points $u,v$ in $\Gamma$ with $d_{\Gamma}(u,v)\geq\sigma$ there exists a proper $\sigma$-separated edge cutset $C_{i}$ with $u$ and $v$ lying in separated components of $\Gamma - C_{i}$.
 		We say the graph is \emph{disjointly strongly edge $\sigma$-separated} if the above cutsets form a partition of the edges.
\end{definition}
There is a more combinatorial condition that implies strong edge separation.
\begin{definition}
 We say a cutset $C$ separates $\{v_{1},v_{2}\}$ and $\{w_{1},w_{2}\}$ if each $v_{i}$ lies in a different component of $\Gamma - C$ to each $w_{j}$.
\end{definition}
\begin{lemma}\label{lem: strong edge sep condition}
    Let $n\geq 2$, and let $\Gamma$ be a graph endowed with the combinatorial metric, such that $girth(\Gamma)\geq 2n$. Suppose that $\Gamma$ is edge $n$-separated with cutsets $\mathcal{C}=\{C_{1},\hdots , C_{m}\}$, and for every pair of vertices $u,v$ in $\Gamma$, and any vertices $u',v'$ with $d_{\Gamma}(u,v)\geq n$ and $d(u,u')=d(v,v')=1$ there exists an $n$-separated cutset $C_{i}$ separating $\{u,u'\}$ and $\{v,v'\}$. Then $\Gamma$ is strongly edge $n$-separated with the same cutsets.
\end{lemma}
\begin{proof}
First note that as $\Gamma$  is edge $n$-separated, it is connected and contains no vertices of degree $1$. Let $u,v$ be two points in $\Gamma$ with $d(u,v)\geq n$. If $u,v$ are vertices, then we are done. Suppose $u$ and $v$ both lie on edges: let $e(u),e(v)$ be the respective edges, and $u_{1},u_{2}$, $v_{1},v_{2}$ the endpoints of $e(u),e(v)$ respectively. If $v$ is a vertex, take $v=v_{1}=v_{2}.$ As $girth (\Gamma)\geq 2n$, without loss of generality $d(u_{1},v_{1})\geq n$: taking $C_{i}$ to be the cutset separating $u_{1},u_{2}$ and $v_{1},v_{2}$, we see that $C_{i}$ separates $u$ and $v$. 
\end{proof}

\begin{definition}[\emph{Weakly vertex separated}]
    Let $\Gamma$ be a finite metric graph, and let $\sigma>0$. We will say that $\Gamma$ is \emph{weakly vertex $\sigma$-separated} if:
    \begin{enumerate}[label=\roman*)]
        \item $\Gamma$ is connected and contains no vertices of degree $1$,
        \item and there exists a collection of $\sigma$-separated vertex cutsets $C_{i}\subseteq V(\Gamma)$ such that $\cup_{i}C_{i}=V(\Gamma)$ and $\vert C_{i}\vert \geq 2$ for each $i$.
       \end{enumerate}

To each vertex cutset $C$ we can assign a partition $\mathcal{P}(C)$ to $\pi_{0}(\Gamma-C)$: we \emph{always require} that such a partition is at least as coarse as connectivity in $\Gamma - C$ and each partition contains at least two elements. The \emph{canonical partition} of $C$ is that induced by connectivity in $\Gamma -C$.
\end{definition} 
\begin{definition}[\emph{Vertex separated}]
    Let $\Gamma$ be a finite metric graph, and let $\sigma>0$. We will say that $\Gamma$ is \emph{vertex $\sigma$-separated} if:
    \begin{enumerate}[label=\roman*)]
        \item $\Gamma$ is connected and contains no vertices of degree $1$,
        \item there exists a collection of $\sigma$-separated vertex cutsets $C_{i}\subseteq V(\Gamma)$ such that $\cup_{i}C_{i}=V(\Gamma)$ and $\vert C_{i}\vert \geq 2$ for each $i$,
        \item for any vertex $v$ and any distinct vertices $w,w'$ adjacent to $v$ there exists a $\sigma$-separated vertex cutset $C_{i}$ such that $w$ and $w'$ lie in separate components of $\Gamma - C_{i}$,
        \item and for any points $u$ and $v$ in $\Gamma$ with $d(u,v)\geq \sigma$, there exists a cutset $C_{i}$ with $u$ and $v$ lying in distinct components of $\Gamma - C_{i}$.
    \end{enumerate}
    	Note that importantly, in general we don't require vertex cutsets to be proper. We say the graph is \emph{disjointly vertex separated} if the above cutsets form a partition of the vertices, and each cutset is proper.

To each vertex cutset $C$ we can assign a partition $\mathcal{P}(C)$ to $\pi_{0}(\Gamma-C)$: we \emph{always require} that such a partition is at least as coarse as connectivity in $\Gamma - C$ and each partition contains at least two elements. The \emph{canonical partition} of $C$ is that induced by connectivity in $\Gamma -C$.
\end{definition}
\begin{remark}
The reason we don't require vertex cutsets to be proper is the following. For edge cut sets we could weaken the definition of edge separated to require a condition similar to $iii)$ above: i.e. that the endpoints of each edge are separated by some cutset. However such a cutset can always be made minimal, and therefore proper, by removing unnecessary edges: the same is not true for vertex cutsets.
\end{remark}

Once again, this definition is not as difficult to verify as it may seem.
\begin{lemma}\label{lem: vertex separated condition}
  Let $n\geq 2$, and let $\Gamma$ be a graph endowed with the combinatorial metric, such that $\Gamma$ is connected, contains no vertices of degree $1$, and $girth(\Gamma)\geq 2 n$. Suppose there exists a collection of $n$-separated vertex cutsets $\mathcal{C}=\{C_{1},\hdots , C_{m}\}$ so that 
  \begin{enumerate}[label=$\roman*)$]
      \item  $\cup_{i}C_{i}=V(\Gamma)$,
      \item  $\vert C_{i}\vert \geq 2$ for each $i$,
      \item for each vertex $v$ and distinct $w$, $w'$ adjacent to $v$ there exists a $n$-separated cutset with $w$ and $w'$ lying in separate components of $\Gamma - C$,
      \item and furthermore that for any pair of vertices $u,v$ with $d_{\Gamma}(u,v)\geq n$ there exists a cutset $C_{i}$ with $u$ and $v$ lying in separate components of $\Gamma -C_{i}$.
  \end{enumerate}
 Then $\Gamma$ is vertex $\sigma$-separated with the collection of cutsets $\mathcal{C}$.
\end{lemma}
\begin{proof}
It suffices to show that for any pair of points $u,v$ with $d(u,v)\geq \sigma$ there exists a cutset $C_{i}$ separating them. If $u$ and $v$ are vertices, then we are finished. Otherwise, let $e(u),e(v)$ be the edges that $u$ and $v$ lie on. Let $u_{1},u_{2}$ and $v_{1},v_{2}$ be the endpoints of $e(u),e(v)$ respectively. If $v$ is a vertex simply take $v_{1}=v_{2}=v$. Then without loss of generality, as $girth(\Gamma)\geq 2n$ and $d(u,v)\geq n$, we have that $d(u_{1},v_{1})\geq n$. Let $C_{i}$ be the cutset separating $u_{1}$ and $v_{1}$: this cutset must also separate $u$ and $v$.
\end{proof}
Finally, we define weighted $\sigma$-separated.

\begin{definition}[Weighted $\sigma$-separated.]
    Let $\sigma>0$ and let $\Gamma$ be an edge $\sigma$-separated graph (respectively strongly edge $\sigma$-separated, weakly vertex $\sigma$-separated, vertex $\sigma$-separated) with $\sigma$-separated cutsets $\mathcal{C}=\{C_{1},\hdots , C_{m}\}$. We call $\Gamma$ \emph{weighted edge $\sigma$-separated} (respectively \emph{strongly edge $\sigma$-separated, weakly vertex $\sigma$-separated, vertex $\sigma$-separated}) if there exists an assignment of positive integers $n(C_{i})$ to the cutsets in $\mathcal{C}$ that solves the \emph{weight equations}: for any edges (respectively edges, vertices, vertices) $\alpha,\beta$ of $\Gamma$,
    $$\sum\limits_{C_{i}\in\mathcal{C}:\alpha\in C_{i}}n(C_{i})=\sum\limits_{C_{i}\in\mathcal{C}:\beta\in C_{i}}n(C_{i}).$$
\end{definition}
Note that though the above equations at first appear to be difficult to solve, we can always find solutions for a graph with an edge (respectively vertex) transitive automorphism group (see Section \ref{subsection: Examples of separated graphs}).

Next we extend these definitions to $CAT(0)$ polygonal complexes. This requires some care to ensure that the subcomplexes we build will actually be separating.
A \emph{polygonal complex} is a $2$-dimensional polyhedral complex and is \emph{regular} if either all polygonal faces are regular polygons.  For a polygonal complex $X$ and a vertex $v$ we define the \emph{link} of $v$, $Lk_{X}(v)$ (or simply $Lk(v)$ when $X$ is clear from context), as the graph whose vertices are the edges of $X$ incident at $v$, and two vertices $e_{1}$ and $e_{2}$ are connected by an edge $f$ in $Lk(v)$ if the edges $e_{1}$ and $e_{2}$ in $X$ are adjacent to a common face $f$. We can endow the link graph with the \emph{angular metric}: an edge $f=(e_{1},e_{2})$ in $Lk(v)$ has length $\alpha$, where $\alpha$ is the angle between $e_{1}$ and $e_{2}$ in the shared face $f$.

We first define the following graph, which appeared in \cite{Ollivier-Wise}. 
\begin{definition}[\emph{Antipodal graph}]
Let $Y$ be a regular non-positively curved polygonal complex. Subdivide edges in $Y$ and add vertices at the midpoints of edges: call these additional vertices \emph{secondary vertices}, and call the other vertices \emph{primary}. Every polygon in $Y$ now contains an even number of edges in its boundary. Construct a graph $\Delta_{Y}$ as follows. Let $V(\Delta_{Y})=V(Y)$ and join two vertices $v$ and $w$ by an edge, labelled $f$, if $v$ and $w$ exist and are antipodal in the boundary of a face $f$ in $Y$: add as many edges as such faces exist. This is the \emph{antipodal graph} for $Y$.
\end{definition}
\begin{remark}
We note that for a secondary vertex $s$ of $Y$, $s$ is a cage graph with edges of length $\pi$. Hence, if $Y$ does not contain any free faces, $Lk_{Y}(s)$ is weighted edge $\pi$-separated, with a single $\pi$-separated cutset $E(Lk_{Y}(s))$.
\end{remark}
Note that as the complex is regular, the edges of $\Delta_{Y}$ pass through the midpoints of edges in $Lk_{Y}(v)$ for vertices $v$.
There is a canonical immersion $\Delta_{Y} \looparrowright Y$; we map a vertex $v$ of $\Delta_{Y}$ to the corresponding vertex of $Y$, and we map an edge $e$ labelled by $f$ to the local geodesic between the endpoints of $e$ lying in the face $f$.

\begin{definition}
 Let $Y$ be a non-positively curved polygonal complex, and let $\Delta$ be one of $Y^{(1)}$ or $\Delta_{Y}$. Assign $\Delta$ an arbitrary orientation, and let $e$ be an oriented edge of $\Delta$. For each $\pi$-separated cutset $C$ in $Lk(i(e))$, choose a set of partitions of $\pi_{0}(Lk(i(e))-C)$, $\{P_{i}(C)\}_{i}$. For $v\in V(\Delta)$, we define 
$$\mathcal{C}_{v}=\{C\;:\;C\mbox{ is a }\pi\mbox{-separated cutset},\;C\subseteq Lk(v)\}.$$
We define
 $$\mathcal{C}(e):=\{C\;:\;C\mbox{ is a }\pi\mbox{-separated cutset},\;e\in C\},$$
and  $$\mathcal{C}=\bigcup_{e\in E^{\pm 1}(\Delta)}\mathcal{C}(e).$$
Similarly we can define
 $$\mathcal{CP}(e):=\bigcup\limits_{C\in\mathcal{C}(e)}\{(C,P_{i}(C))\}_{i},$$
 and $$\mathcal{CP}=\bigcup_{e\in E^{\pm 1}(\Delta)}\mathcal{CP}(e).$$

 \end{definition}
The following is extremely similar to the `splicing' of Manning \cite{Manning10}: we will use this for a similar purpose to that of \cite{Cashen-Macura2011}.
\begin{definition}[Equatable partitions]
 Let $Y$ be a non-positively curved polygonal complex, and let $\Delta$ be one of $Y^{(1)}$ or $\Delta_{Y}$. Let $v,w$ be two vertices of $\Delta$ connected by an oriented edge $e$, so that $v=i(e)$ and $w=t(e)$. Let $C_{v}$ be a $\pi$-separated cutset in $Lk(v)$ with choice of partition $P_{v}$ and $C_{w}$ be a $\pi$-separated cutset in $Lk(w)$ with choice of partition $P_{w}$.
 
 Let $v'$, $w'$ be points on $e$ in an $\epsilon$-neighbourhood of $v$, $w$ respectively, so that there are canonical mappings
 \begin{equation*}
     \begin{split}
         i_{v}&:St(v')\hookrightarrow Lk(v),\\
         i_{w}&:St(w')\hookrightarrow Lk(w),\\
         \phi&:St(v')\xrightarrow{\cong}St(w').
     \end{split}
 \end{equation*}
 Therefore we have induced mappings
 \begin{equation*}
     \begin{split}
         \overline{i}_{v}&:St(v')-v'\hookrightarrow Lk(v)-C_{v},\\
         \overline{i}_{w}&:St(w')-w'\hookrightarrow Lk(w)-C_{w},\\
         \overline{\phi}&:St(v')-v'\xrightarrow{\cong}St(w')-w'.
     \end{split}
 \end{equation*}
For $u=v,w$ let $\mathcal{P}_{u}$ be the set of partitions of $\pi_{0}(Lk(u)-C_{u})$, and let $\mathcal{P}_{u'}$ be the set of partitions of $\pi_{0}(St(u')-u')$. There are induced maps 
  \begin{equation*}
     \begin{split}
         \iota_{v}&:\mathcal{P}_{v}\rightarrow \mathcal{P}_{v'},\\
         \iota_{w}&:\mathcal{P}_{w}\rightarrow \mathcal{P}_{w'},\\
         \psi&:\mathcal{P}_{v'}\hookdoubleheadrightarrow \mathcal{P}_{w'}.
     \end{split}
 \end{equation*}

 We say that $(C_{v},P_{v})$ and $(C_{w},P_{w})$ are \emph{equatable along $e$}, written $$(C_{v},P_{v})\sim_{e}(C_{w},P_{w})$$
 if $$\psi(\iota_{v}(P_{v}))=\iota_{w}(P_{w}).$$

Note that this also defines an equivalence relation on $\mathcal{CP}(e)$: for $(C,P),(C',P')\in \mathcal{CP}(e)$, we write $$(C,P)\approx_{e} (C',P')$$ if 
$$\iota_{v}(P)=\iota_{v}(P').$$
This defines an equivalence relation on $\mathcal{CP}(e)$, and so defines an equivalence class $[C,P]_{e}$. We define $[C,P]_{e^{-1}}$ to be the equivalence class of cutset partitions in $\mathcal{CP}(e^{-1})$ equatable to $(C,P)$ along $e$: by definition this is independent of choice of $(C',P')\in [C,P]_{e}$.
\end{definition}
These constructions are designed so that we can `splice' the local cutsets along each edge. Though this definition is somewhat complicated, note the following remark.

\begin{remark}
Let $e$, $v,w$, $C_{v},C_{w}$ be as above. If both $C_{v},C_{w}$ are proper with canonical partitions $P_{v},P_{w}$, then $(C_{v},P_{v})\sim_{e}(C_{w},P_{w})$. This follows as the induced partitions of $St(v')-v'$ and $St(w')-w'$ are just the partitions induced by connectivity, and by properness every element of the induced partition of $St(v')-v'$ (respectively $St(w')-w'$) contains a unique vertex.

Similarly, if $C_{1},C_{2}\in\mathcal{C}(e)$ are proper, with canonical partitions $P_{1},P_{2}$, then $(C_{1},P_{1})\approx_{e}(C_{2},P_{2})$.
\end{remark}
\begin{definition}[\emph{Gluably $\sigma$-separated}]
Let $Y$ be a non positively curved polygonal complex. We call $Y$ \emph{gluably edge $\sigma$-separated} (respectively \emph{gluably  (weakly) vertex $\sigma$-separated}) if :
\begin{enumerate}[label=$\roman*)$]
\item $Y$ is regular (respectively $Y$ is allowed \textbf{not to be regular})
    \item the link of every vertex in $Y$ is edge (respectively (weakly) vertex) $\sigma$-separated,
    \item for every $\pi$-separated cutset $C$ in $Lk(v)$ there exists a series of partitions $\{P_{i}(C)\}$ of $\pi_{0}(Lk(v)-C)$ such that for any distinct pair of points $x,y\in Lk(v)$ separated by $C$, $x$ and $y$ are separated by some $P_{i}(C),$ 
    \item and there exists a strictly positive integer solution to the \emph{gluing equations}: letting $\Delta=\Delta_{Y}$ (respectively $\Delta=Y^{(1)}$) we can assign a positive integer $\mu (C,P)$ to every pair $$(C,P)\in\mathcal{CP}:=\bigcup\limits _{e\in E^{\pm}(\Delta )}\mathcal{CP}(e)$$ such that for every edge $e$ of $\Delta$ and every $(C,P)\in \mathcal{CP}(e)$:
\begin{equation*}
	    \sum\limits_{(C',P')\in [C,P]_{e}} \mu(C',P')=\sum\limits_{(C',P')\in [C,P]_{e^{-1}}} \mu(C',P').
	\end{equation*}
\end{enumerate}
\end{definition}

\begin{definition}[\emph{Gluably $\sigma$-separated}]
Let $Y$ be a non positively curved polygonal complex. We call $Y$:
\begin{enumerate}[label=$\roman*)$]
    \item \emph{gluably weakly $\sigma$-separated} if it is gluably weakly vertex $\sigma$-separated,
    \item and \emph{gluably $\sigma$-separated} if it is gluably edge or gluably vertex $\sigma$-separated.
\end{enumerate}
\end{definition}
\begin{remark}
Again, note that in the definition of a gluably (weakly) vertex $\sigma$-separated complex, \textbf{we do not require that the complex $Y$ is regular}. If the link of each vertex in the complex $Y$ is disjointly $\sigma$-separated, then we can solve the gluing equations by taking only the canonical partition $P(C)$ for each cutset $C$, and setting $\mu(C,P(C))=1$ for all cutsets $C$, so that $Y$ is gluably $\sigma$-separated.
\end{remark}

%----------------------------------------------------------------------------------------
%   REMOVING CUT EDGES
%----------------------------------------------------------------------------------------	
	\subsection{Removing cut edges}\label{subsection: Removing cut edges and producing two sided cut sets}
	We now show the existence of cut edges is not too much of an issue.

	\begin{lemma}\label{lem: removing cut edges}
	Let $G$ be a group acting properly discontinuously and cocompactly on a simply connected $CAT(0)$ polygonal complex $X$, such that the link of every vertex in $X$ is connected. There exists a simply connected $CAT(0)$ polygonal complex $X'$ such that $G$ acts properly discontinuously and cocompactly on $X'$ and the link of any vertex $v'$ in $X'$ is a subgraph of $Lk(v)$ for some vertex $v$ in $X$. Furthermore for any vertex $v$ of $X'$, either $Lk_{X'}(v)$ is connected and contains no cut edges, or $Lk_{X'}(v)$ is disconnected.
	\end{lemma}
	\begin{proof}
	First note that we can assume that $X$ contains no vertices of degree $1$ in its links. Let $Y=G\backslash X$, and let $v_{0},\hdots ,v_{m}$ be the vertices of $Y$. 
	
	Let $v$ be a vertex in $X$, and suppose there exists a cut edge $f$ in $Lk(v)$. Let $e_{1}$ and $e_{2}$ be the endpoints of $f$. Suppose that, in $X$, the endpoints of $e_{1}$ are $v$ and $w$. Construct a new complex $X'$ as follows: Let $v_{1}$ and $v_{2}$ be two copies of $v$ and connect these vertices to $w$ with the edges $e_{1}^{1}$ and $e_{2}^{2}$ respectively. Since $f$ is a cut edge in $Lk(v)$ there is a canonical way to attach edges and faces to $v_{1}$ and $v_{2}$ that agrees with the connected components of $Lk(v) - e_{1}$.
	
	Now, we assume that $f$ is attached to $v_{1}$. Then the face $f$ is a free face, which we can push in to remove the vertex of degree $1$, $e_{1}^{1}$ in $Lk(v_{1})$, so that $Lk(v_{1})$ and $Lk(v_{2})$ are connected subgraphs of $Lk(v)-f$, and the links of any other vertices $w$ incident to the face $f$ are transformed to a proper subgraph of $Lk(w)$ with the edge $f$ removed.
	
	We can repeat this process finitely many times, applied to the set of vertices $Gv_{i}$ each time, to find the $CAT(0)$ polygonal complex $X'$ desired. 
	\end{proof}

%----------------------------------------------------------------------------------------
%   Examples of separated graphs
%----------------------------------------------------------------------------------------	
	\subsection{Examples of separated graphs}\label{subsection: Examples of separated graphs}
	
Our definitions of weighted $\sigma$-separated graphs required assigning weights to cutsets such that certain equations hold. In this subsection we prove that as long as the automorphism group of a graph is transitive on vertices (or edges, depending on whether cutsets are formed of vertices or edges), then these equations can always be solved. Note that $Aut(\Gamma)$ is the group of automorphisms of $\Gamma$ as a metric graph.
\begin{lemma}\label{lem: vertex transitive automorphism group can solve equations}
  Let $\sigma>0$ and let $\Gamma$ be (weakly) vertex $\sigma$-separated. If $Aut(\Gamma)$ is vertex transitive then $\Gamma$ is weighted (weakly) vertex $\sigma$-separated.  
\end{lemma}
\begin{proof}
Assume that $\Gamma$ is vertex $\sigma$-separated, with $\sigma$-separated vertex cutsets $\mathcal{C}=\{C_{1},\hdots, C_{n}\}$. The proof is similar for weakly separated graphs. Let $H=Aut(\Gamma)$. For each $C\in\mathcal{C}$, let $$H(C):=\{\gamma C\;:\;\gamma\in H\},$$ counted {\bf{with}} multiplicity, i.e. if $\gamma_{1}C=\gamma_{2}C$ and $\gamma_{1}\neq_{H}\gamma_{2}$, then both $\gamma_{1}C,\gamma_{2}C$ appear in $H(C)$. Note that for every $C'\in H(C)$, $C'$ is a $\sigma$-separated vertex cutset.

Fix some vertex $v\in C$, and let $w\in V(\Gamma)$ be any vertex. Since $H$ acts vertex transitively, there exists $h\in H$ such that $h v=w$. Therefore 
$$\{\gamma\in H \;:\;v\in \gamma C\}=\{\gamma\in H \;:\;w\in h\gamma C\}=\{h^{-1}\gamma'\in H \;:\;w\in \gamma' C\},$$
and therefore $$\vert\{\gamma\in H \;:\;v\in \gamma C\}\vert=\vert\{\gamma\in H \;:\;w\in \gamma C\}\vert.$$

Let $$\tilde{\mathcal{C'}}:=\bigsqcup\limits_{C\in\mathcal{C'}}H(C),$$
again with multiplicity. By the above, it follows that for any two vertices $v,w\in V(\Gamma)$,
$$\vert\{C\in \tilde{\mathcal{C}'} \;:\;v\in  C\}\vert=\vert\{C\in \tilde{\mathcal{C}'} \;:\;w\in  C\}\vert.$$
Let $\mathcal{C}'$ be the underlying set of $\tilde{\mathcal{C}'}$, and for $C\in \mathcal{C}',$ let $$n(C)=\vert\{C'\in\tilde{\mathcal{C}'}\;:\;C=C'\} \vert,$$
i.e. $n(C)$ is the multiplicity of $C$ in $\tilde{\mathcal{C}'}.$ It is easily seen that the above weights solve the gluing equations.

As $\mathcal{C}\subseteq\mathcal{C}'$, it follows that $\Gamma$ is vertex separated with respect to these cutsets: by the above argument it follows that $\Gamma$ is weighted vertex $\sigma$-separated with cutsets $\mathcal{C}'$.
\end{proof}
Similarly, we can prove the following.
\begin{lemma}\label{lem: edge transitive automorphism group can solve equations}
     Let $\sigma>0$ and let $\Gamma$ be (strongly) edge $\sigma$-separated. If $Aut(\Gamma)$ is edge transitive then $\Gamma$ is weighted (strongly) edge $\sigma$-separated.  
\end{lemma}

%----------------------------------------------------------------------------------------
%   Examples of separated graphs
%----------------------------------------------------------------------------------------	
	\subsection{Examples of solutions of the gluing equations}\label{subsection: Examples of solutions of the gluing equations}

Recall that we call an edge cutset $C$ \emph{proper} if the endpoints of every edge $e$ in $C$ lie in separate components of $\Gamma-C$, and a vertex cutset $C$ \emph{proper} if for every $v\in C$, the vertices adjacent to $v$ each lie in separate components of $\Gamma-C$.
\begin{lemma}\label{lem: solving gluing equations for minimal edge cutsets}
  Let $Y$ be a regular non-positively curved complex and suppose the link of each vertex is weighted edge $\pi$-separated. There exists a system of strictly positive weights that solve the gluing equations for $Y$.
\end{lemma}
\begin{proof}
Since edge cutsets are proper, any two cutsets are equatable along a shared edge. Therefore we may associate to each cutset $C$ exactly one partition $P(C)$, namely that of connectivity in $\Gamma-C$. In particular for any oriented edge $e\in E^{\pm}(\Delta_{Y})$ and any $(C,P(C))\in \mathcal{CP}(e),$ $[C,P(C)]_{e}=\mathcal{CP}(e).$

First, note that for an oriented edge $e$ of $\Delta_{Y}$, and $v=i(e)$, $\mathcal{C}(e)=\mathcal{C}(e)\cap \mathcal{C}_{v}$. Since the link of each vertex in $Y$ is weighted edge $\pi$-separated, for each vertex $v\in Y$ there exists a positive integer $N_{v}>0$ and a system of strictly positive weights $n_{v}(C)$ for $C\in \mathcal{C}_{v}$ such that for any edge $e$ in $Lk_{Y}(v)$,

$$\sum\limits_{C\in \mathcal{C}(e)}n_{v}(C)=\sum\limits_{C\in \mathcal{C}(e)\cap \mathcal{C}_{v}}n_{v}(C)=N_{v}.$$
Let $M=\prod_{v\in V(Y)}N_{v},$ and for a cutset $C\in\mathcal{C}_{v},$ define $m(C)=Mn_{v}(C)\slash N_{v}.$ It follows that for an edge $e$ in $Lk_{Y}(v)$, 
$$\sum\limits_{C\in \mathcal{C}(e)}m(C)=\frac{M}{N_{v}}\sum\limits_{C\in \mathcal{C}(e)}n_{v}(C)=\frac{M}{N_{v}}N_{v}=M.$$
Finally, taking $\mu(C,P(C))=m(C)$, these weights immediately solve the gluing equations.
\end{proof}

Similarly, we can prove the following.

\begin{lemma}\label{lem: solving gluing equations for proper vertex cutsets}
  Let $Y$ be a non-positively curved complex, such that the link of each vertex is weighted vertex $\pi$-separated, and every cutset is proper. There exists a system of strictly positive weights that solve the gluing equations for $Y$.
\end{lemma}
%----------------------------------------------------------------------------------------
%   CONSTRUCTING HYPERGRAPHS
%----------------------------------------------------------------------------------------
\subsection{Hypergraphs in \texorpdfstring{$\mathbf{\pi}$}{pi}-separated polygonal complexes}\label{subsection: Constructing hypergraphs in polygonal complexes}
We now begin to construct our separating subcomplexes.
Suppose $X$ is a simply connected $CAT(0)$ polygonal complex, and $G$ acts properly discontinuously and cocompactly on $X$, so that $G\backslash X$ is (weakly) gluably $\pi$-separated. If $G\backslash X$ is gluably edge $\pi$-separated, let $\Delta=\Delta_{G\backslash X}$, and if it is (weakly) gluably vertex $\pi$-separated, let $\Delta=(G\backslash X)^{(1)}$. Assign an arbitrary orientation to $\Delta$. 
Recall that for an oriented edge $e$ of $\Delta$, we let $\mathcal{C}(e)=\{C\in\mathcal{C}\;:\;e\in C\}$ (note that for any oriented edge $e$, $\mathcal{C}(e)$ is non-empty, as $G\backslash X$ is gluably edge $\pi$-separated). 
 For every vertex $v$ and $\pi$-separated cutset $C$ in $Lk(v)$ let $\{P_{i}(C)\}$ be the required set of partitions of $\pi_{0}(Lk(v)-C)$, and let
 $$\mathcal{CP}(e)=\bigcup\limits_{C\in\mathcal{C}(e)}\{(C,P_{i}(C)\}_{i}.$$
Let $$\mathcal{CP}=\bigcup\limits_{e\in E^{\pm}(\Delta)}\mathcal{CP}(e).$$
By assumption, we can assign positive integer weights $\mu(C,P)$ to each cutset $(C,P)\in\mathcal{CP}$ so that for every edge $e$ of $\Delta$:

 \begin{equation*}
	    \sum\limits_{(C',P')\in [C,P]_{e}} \mu(C',P')=\sum\limits_{(C',P')\in [C,P]_{e^{-1}}} \mu(C',P').
	\end{equation*}
We now construct a second graph $\Sigma$ as follows. Let 
$$V(\Sigma)=\bigsqcup\limits_{(C,P)\in\mathcal{CP}}\{u_{(C,P)}^{1},\hdots ,u_{(C,P)}^{\mu(C,P)}\}.$$
The gluing equations imply that for each positively oriented edge $e$ of $\Delta$ and each equivalence class $[C,P]_{e}\subseteq \mathcal{CP}(e)$ there exists a bijection $$\phi_{e}:\hspace*{- 10pt}\bigsqcup\limits_{(C',P')\in [C,P]_{e}}\hspace*{- 10pt}\{u_{(C',P')}^{1},\hdots ,u_{(C',P')}^{\mu(C',P')}\}\rightarrow\hspace*{- 10pt} \bigsqcup\limits_{(C',P')\in [C,P]_{e^{-1}}}\hspace*{- 10pt}\{u_{(C',P')}^{1},\hdots ,u_{(C',P')}^{\mu(C',P')}\}.$$

For each positively oriented edge $e$ and each equivalence class $[C,P]_{e}$ of $\mathcal{CP}(e)$ choose such a bijection, $\phi_{e}$, and add the oriented edges $$\{(u_{(C',P')}^{i},\phi_{e}(u_{(C',P')}^{i}))\;:\;(C',P')\in[C,P]_{e}, 1\leq i\leq \mu(C',P')\}.$$ 

Note that for each $(C,P)\in\mathcal{CP}$, $Lk_{\Sigma}(u_{(C,P)}^{i})$ is isomorphic to $C$ as labelled oriented graphs. Furthermore, each edge in $\Sigma$ labelled by $e$ connects two vertices of the form $u_{(C,P)}^{i}$ $u_{(C',P')}^{j}$ with $(C,P)\sim_{e}(C',P')$, i.e. every edge connects vertices with equatable partitions along that edge. There is an immersion $\Sigma\looparrowright\Delta$ that sends $u_{(C,P)}^{i}$ to the vertex $v_{C}$ such that $C\subseteq Lk_{\Delta}(v_{C})$ and maps an edge labelled by $e$ to the edge $e$ in $\Delta$.

Let $\Sigma_{1},\hdots ,\Sigma_{m}$ be the connected components of $\Sigma$, and let $\underline{\Lambda}^{1},\hdots , \underline{\Lambda}^{m}$ be the images of these graphs in $G\backslash X$ under the immersion $$\Sigma_{i}\looparrowright\Delta\looparrowright G\backslash X.$$ Note that each $\underline{\Lambda}^{i}$ is locally geodesic as the cut sets are $\pi$-separated in $G\backslash X$. 

\begin{definition}
If $G\backslash X$ is gluably edge $\pi$-separated, a lift of $\underline{\Lambda}^{i}$ from to the $CAT(0)$ complex $X$ is called a \emph{edge hypergraph in $X$}, and otherwise it is a \emph{vertex hypergraph in $X$}.
\end{definition}

Note that hypergraphs come with two pieces of information at each vertex $v$ in $X$: the cutset $C$ and partition $P$. We say $\Lambda$ \emph{passes through} the above objects.
 \begin{remark}
Note that in the above construction for every vertex $v$, every $\pi$-separated edge cutset $C$ in $Lk(v)$ and chosen partition $P$ of $\pi_{0}(Lk(v)-C)$, and every lift $\tilde{v}$ of $v$, there exists a hypergraph passing through $(C,P)$ in $Lk_{X}(\tilde{v})$.
\end{remark}
\begin{figure}[H]
    \centering
    \includegraphics{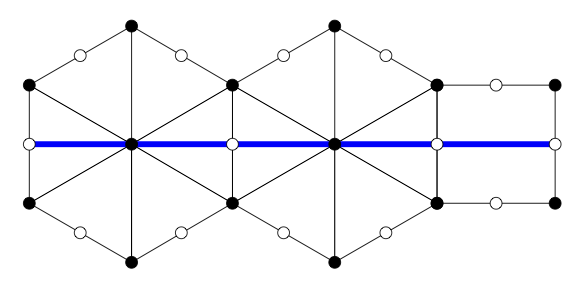}
    \caption{Example of subsection of an edge hypergraph.}
\end{figure}
Importantly, our construction ensures that for any hypergraph $\Lambda$, $Stab(\Lambda)$ acts properly discontinuously and cocompactly on $\Lambda$.

%----------------------------------------------------------------------------------------
%   Hypergraphs are separating
%----------------------------------------------------------------------------------------
\subsection{Hypergraphs are separating}\label{subsection: Hypergraphs are separating}
We now analyse the structure of the hypergraphs, and show they are in fact separating.
\begin{lemma}\label{lem: edge hypergraphs are leafless convex}
	Let $G\backslash X$ be a simply connected (weakly) gluably $\pi$-separated $CAT(0)$ polygonal complex, and let $\Lambda$ be a hypergraph in $X$. Then $\Lambda$ is a leafless convex tree.
\end{lemma}

\begin{proof}
We prove this for edge hypergraphs: the argument is similar for vertex hypergraphs.
	For each $i$, as the cutsets are $\pi$-separated, the image of $\Sigma_{i}\looparrowright G\backslash X$ is locally geodesic: therefore $\Lambda$ is locally geodesic in $X$. As $X$ is $CAT(0)$, local geodesics are geodesic, and geodesics are unique, so that $\Lambda$ is a convex tree.
	Since $\vert C\vert\geq 2$ for any $v\in V(\Delta_{G\backslash X})$ and $C\in\mathcal{C}_{v}$, $\Lambda$ contains no primary vertices of degree $1$. Similarly, as there are no vertices of degree $1$ in the link of a primary vertex in $X$, there are no cut edges in the link of a secondary vertex and so every cut set contains at least two edges. It follows that edge hypergraphs are leafless.
\end{proof}
	\begin{definition}
		 Let $\Lambda_{i}$ be a hypergraph in $X$ and $x,y\in X$ be distinct points in $X$. We say $\Lambda_{i}$ \emph{separates} $x$ and $y$ if $x$ and $y$ lie in distinct components of $X-\Lambda_{i}$.
		 We write $\#_{\Lambda}(x,y)$ for the number of edge (or vertex) hypergraphs separating $x$ and $y$.
	\end{definition}
	
	We now consider separating points: we prove the following lemma.
	We call a path $\gamma$ \emph{transverse} to $\Lambda$ if $\vert \gamma\cap \Lambda\vert =1$
	\begin{lemma}\label{lem: separation condition}
	    Let $\Lambda$ be a hypergraph in $X$, and $\gamma=[p,q]$ be a geodesic transverse to $\gamma$. If $\gamma\cap\Lambda=\{x\}$ and $p$ and $q$ lie in different elements of the partition of $Lk(x)-\Lambda$, then $p$ and $q$ lie in different components of $X-\Lambda$.
	\end{lemma}
	Before we prove this, we need to define some technology.

	\begin{definition}[Hypergraph retraction]
	 Let $X$ be a $CAT(0)$ space and $\Lambda$ a hypergraph in $X$. The projection map $$\pi_{\Lambda}:X\rightarrow\Lambda$$ maps every point in $X$ to its nearest point in $\Lambda$. Since $\Lambda$ is convex, this map is a deformation retraction. That is, we have a homotopy $\pi_{\Lambda}^{\delta}$ from the identity to $\pi_{\Lambda}.$
	\end{definition}
	\begin{definition}[$\Lambda$-balanced paths]
	  Let $\Lambda$ be a hypergraph in $X$ and $\epsilon>0$. For points $p,q$ lying in the same component of $X-\Lambda$, a \emph{$(\Lambda,\epsilon)$-balanced path} from $p$ to $q$ is a path $\sigma$ starting at $p$ and ending at $q$ such that:
	  
	 \begin{enumerate}[label=\roman*)]
	     \item $\sigma \subseteq \mathcal{N}_{\epsilon}(\Lambda)-\Lambda,$
	     \item and for any $x\in \Lambda$, $\vert(\pi_{\Lambda})^{-1}(x)\cap \sigma\vert$ is even, and in particular finite.
	 \end{enumerate}
	 
Given such a path, we define $$m_{\Lambda}(\sigma)=\frac{1}{2}\sum \limits_{e\in E(\Lambda\cap \pi_{\Lambda}(\sigma))}\max\limits_{y\in e} \vert(\pi_{\Lambda})^{-1}(y)\cap \sigma\vert. $$
	\end{definition}
By considering the retraction map, we see that such paths exist.	
	\begin{lemma}\label{lem: existence of balanced paths}
	    Let $\Lambda$ be a hypergraph in $X$, $\epsilon>0$, and let $p,q\in \mathcal{N}_{\epsilon}(\Lambda)$ be points lying in the same component of $X-\Lambda$. There exists a $(\Lambda,\epsilon)$-balanced path between them.
	\end{lemma}
	\begin{proof}
Let $\gamma$ be a path from $p$ to $q$ not intersecting $\Lambda$. By taking $\delta$ close to $1$, we have that $\sigma_{\delta}:=\pi_{\Lambda}^{\delta}(\gamma)\subseteq \mathcal{N}_{\epsilon}(\Lambda)-\Lambda$. Since $\sigma_{\delta}$ maps to a loop in $\Lambda$, which is a tree, it immediately follows that by taking $\delta$ close to $1$, and after a small homotopy, for any $y\in \Lambda$, $\vert(\pi_{\Lambda})^{-1}(y)\cap \sigma_{\delta}\vert$ is even, and in particular finite.
	\end{proof}
	
We can now prove Lemma \ref{lem: separation condition}.
\begin{proof}[Proof of Lemma \ref{lem: separation condition}]
Throughout we choose $\epsilon>0$ sufficiently small so that for any pair of vertices $v,w$ of $X$, $\mathcal{N}_{\epsilon}(v)\cap\mathcal{N}_{\epsilon}(w)=\emptyset$.

Let $P$ be the partition of $Lk(x)$ through which $\Lambda$ passes.
  Let $P_{1}$ be the element of $P$ containing $p$ and $P_{2}$ the element of $P$ containing $q$. We may assume that $p,q\in \mathcal{N}_{\epsilon}(\Lambda)$: otherwise, choose $w_{i}$ lying in $P_{i}$ such that $w_{i}\in \mathcal{N}_{\epsilon}$. Then $w_{1}$ and $p$ lie in the same component of $X-\Lambda$ and $w_{2}$ and $q$ lie in the same component of $X-\Lambda$. Suppose $p$ and $q$ lie in the same component of $X-\Lambda$. We first choose $p'\in P_{1}$, $q'\in P_{2}$ and $\sigma $ a $\Lambda$-balanced path between $p'$ and $q'$ so that the pair $(m_{\Lambda}(\sigma),l(\sigma))$ is minimal by lexicographic ordering amongst all such $p',q',\sigma$. We induct on $(m_{\Lambda}(\sigma),l(\sigma))$ by lexicographic ordering.
  
  If $m_{\Lambda}(\sigma)=1$, then $\sigma$ passes along exactly one edge $e$ of $\Lambda$: it follows that the partitions of $X-\Lambda$ at the endpoints of $e$ are not equatable along $e$, or that $P_{1}=P_{2}$, a contradiction.
  
  Otherwise $ m_{\Lambda}(\sigma)=m\geq 2$. Suppose the first and last edges of $\Lambda$ traversed by $\sigma$ are the same edge $e$. Note that we may always travel along $\sigma$ to put ourselves in the situation assumed above: this is analogous to the classical situation of pushing to a leaf in a tree for graph theory arguments. In particular, if this is not true, then move along $\sigma$, starting at $q'$, until we return to $\mathcal{N}_{\epsilon}(x)$. Let $s$ be the point we reach in $\mathcal{N}_{\epsilon}(x)$. If $s$ is in the same component as $q'$ in $P$ then $m_{\Lambda}(\sigma\vert_{[p',s]})\leq  m_{\Lambda}(\sigma)$, and $l(\sigma\vert_{[p',s]})< l(\sigma)$, so that $$(m_{\Lambda}(\sigma\vert_{[p',s]}),l(\sigma\vert_{[p',s]}))<(m_{\Lambda}(\sigma),l(\sigma)),$$ a contradiction as $p',q',\sigma$ where chosen so this pair was minimal. If $s$ is in the same component as $p'$ in $P$, and is not equal to $p'$, then $m_{\Lambda}(\sigma\vert_{[s,q']})\leq  m_{\Lambda}(\sigma)$, and $l(\sigma\vert_{[s,q']})< l(\sigma)$, so that again $$(m_{\Lambda}(\sigma\vert_{[s,q']}),l(\sigma\vert_{[s,q']}))<(m_{\Lambda}(\sigma),l(\sigma))$$ a contradiction.  Therefore if $s \neq p'$, then $s$ lies in a different component to $q'$ in $P$: we have $(m_{\Lambda}(\sigma\vert_{[s,q']}),l(\sigma\vert_{[s,q']}))<(m_{\Lambda}(\sigma),l(\sigma))$, and hence by induction $s$ must lie in a separate component of $X-\Lambda$ to $q'$, a contradiction as $q'$ is connected to $s$ by a path not intersecting $\Lambda$. 
  Therefore by induction we have that $s=p'$.
  
  Let $y$ be the endpoint of $e$ distinct from $x$, let $\alpha$ be the the point obtained by pushing $p'$ along $\sigma $ to $\mathcal{N}_{\epsilon}(y)$, and similarly $\beta $ be the the point obtained by pushing $q'$ along $\sigma $ to $\mathcal{N}_{\epsilon}(y)$. Let $\sigma'$ be the subpath of $\sigma$ connecting $\alpha$ and $\beta$.
  
  If $\alpha$ and $\beta$ lie in the same component of the partition of $Lk(y)-\Lambda$, then the partitions are not equatable along $e$, a contradiction. Otherwise $(m_{\Lambda}(\sigma'),l(\sigma '))< (m_{\Lambda}(\sigma),l(\sigma))$, and so by induction $\alpha$ and $\beta$ lie in distinct components of $X-\Lambda$.
  As $p'$ is connected to $\alpha$ by a path not intersecting $\Lambda$, and $q'$ to $\beta$, we see that $p'$ and $q'$ lie in distinct components. Since $p$ is connected to $p'$ and $q$ is connected to $q'$ by a path not intersecting $\Lambda$, the result follows.
\end{proof}

%----------------------------------------------------------------------------------------
%   Hypergraph stabilizers and wallspaces
%----------------------------------------------------------------------------------------
\subsection{Hypergraph stabilizers and wallspaces}\label{subsection: Hypergraph stabolizers and wallspaces}	
We now want to use the construction of a cube complex dual to a system of walls, as found in \cite{Hruska-Wise}. Their definition of a wallspace is more general than the one we require, and so we restrict to the case that $X$ is endowed with a metric.
\begin{definition}[\emph{Walls}]
    Let $X$ be a metric space. A \emph{wall} is a pair $\{U,V\}$ such that $X=U\cup V$. The \emph{open halfspaces} associated to the wall are $U-(U\cap V)$ and $V-(U\cap V)$. We say a wall \emph{betwixts} a point $x$ if $x\in U\cap V$, and \emph{separates} the points $x,y$ if $x$ and $y$ lie in distinct open halfspaces. We write $\#(x,y)$ for the number of walls separating $x$ and $y$.
\end{definition}

\begin{definition}[\emph{Wallspace}]
 A \emph{wallspace} is a pair $(X,\mathcal{W})$, where $X$ is a connected metric space and $\mathcal{W}$ is a collection of walls in $X$ such that;
 \begin{enumerate}[label=\roman*)]
     \item for any $x\in X$, finitely many walls in $\mathcal{W}$ betwixt $x$,
     \item for any $x,y\in X$, $\#(x,y)<\infty$,
     \item and there are no duplicate walls that are genuine partitions.
 \end{enumerate}
 We say a group $G$ \emph{acts} on a wallspace $(X,\mathcal{W})$ if $G$ acts on $X$, and $G\cdot\mathcal{W}=\mathcal{W}$.
\end{definition}
\begin{definition}[\emph{$\Lambda$ walls}]
    Let $\Lambda$ be a vertex or edge hypergraph in $X$, with disjoint components $X-\Lambda=\{U^{i}_{\Lambda}\}_{i}$. For each $U_{\Lambda}^{i}$, let $V_{\Lambda}^{i}=X-\overline{U_{\Lambda}^{i}}$.
The \emph{set of $\Lambda$ walls} is the set

$$\mathcal{W}_{\Lambda}=\bigg\{\{\overline{U_{\Lambda}^{i}},\overline{V_{\Lambda}^{i}}\}\;:\;U_{\Lambda}^{i}\mbox{ a component of }X-\Lambda\bigg\}.$$
The \emph{hypergraph wallspace} is the set of walls 
$$\mathcal{W}=\cup_{\Lambda}\mathcal{W}_{\Lambda},$$
where we remove any duplicate walls.
\end{definition}

We now show that the pair $(X,\mathcal{W})$ is a wallspace. There are several easy but technical steps to this.

\begin{lemma}\label{lem: wall stabilizers are cocompact}
  Let $H_{\Lambda}=Stab_{G}(\Lambda)$, and for any $i$, let $H_{\Lambda,i}=Stab_{H_{\Lambda}}(U_{\Lambda}^{i}).$ Then $H_{\Lambda, i}$ acts cocompactly on $\partial U_{\Lambda}^{i}.$ 
\end{lemma}
This Lemma follows immediately from the proof of \cite[Theorem 2.9]{Hruska-Wise}. We include the argument here for completeness. For a set $A$ in a metric space $(X,d)$, we define the \emph{frontier of $A$} as the set $\partial_{f}A=\{x\in X\;\vert 0<d(x,A)\leq 1\}.$
\begin{proof}
Note that $H_{\Lambda}$ acts cocompactly on $\Lambda$ and so on $\partial_{f}\Lambda$. Furthermore $H_{\Lambda}$ preserves the partition of $\partial_{f}\Lambda$ into $U_{\Lambda}^{i}\cap \partial_{f}\Lambda$. Hence $H_{\Lambda,i}$ acts properly discontinuously and cocompactly on $\partial_{f}U_{\Lambda}^{i}$, and therefore on $\partial U_{\Lambda}^{i}$.
\end{proof}

\begin{lemma}\label{lem: finitely many wall orbits}
There are finitely many $G$-orbits of walls in $\mathcal{W}$.
\end{lemma}
\begin{proof}
There are finitely many $G$-orbits of hypergraphs $\Lambda$, and there are finitely many $H_{\Lambda}$ orbits of $U_{\Lambda}^{i}$. The result follows.
\end{proof}

\begin{lemma}\label{lem: W is a wallspace}
    The pair $(X,\mathcal{W})$ is a wallspace.
\end{lemma}
\begin{proof}
First we prove finitely many walls betwixt points. Since the set of walls is acted upon cofinitely by $G$, and each wall has a cocompact stabilizer, this follows immediately. In a similar manner we can observe $\#(x,y)<\infty$ for any $x$ and $y$.
\end{proof}

Therefore we have constructed a wallspace for $X$. 
\begin{lemma}\label{lem: wallspace separation from hypergraph separation}
    Let $\mathcal{W}$ be constructed as above. Then $\#(x,y)\geq\#_{\Lambda}(x,y)$.
\end{lemma}
\begin{proof}
Note that if a hypergraph $\Lambda$ separates $x$ and $y$, by taking $i$ such $x\in U_{\Lambda}^{i}$, it follows that $W_{\Lambda}^{i}$ separates $x$ and $y$. The result follows.\end{proof}

Next we discuss transverse walls.
\begin{definition}[\emph{Transverse}]
    Two walls $W=\{U,V\}$ and $W'=\{U',V'\}$ are \emph{transverse} if each of the intersections $U\cap U', U\cap V', V\cap U', V\cap V'$ are nonempty.
\end{definition}
There is an easier formulation for this definition.
\begin{lemma}
    Two walls $W_{\Lambda}^{i},W_{\Lambda'}^{j}$ are transverse if and only if $\partial U_{\Lambda}^{i}\cap \partial U_{\Lambda'}^{j}$ is non-empty. In particular the walls are transverse only if $\Lambda\cap \Lambda'$ is non-empty.
\end{lemma}
Using this we can now move on to cubulating groups acting on such complexes.
%----------------------------------------------------------------------------------------
%   CUBULATING GROUPS ACTING ON POLYGONAL COMPLEXES
%----------------------------------------------------------------------------------------
\subsection{Cubulating groups acting on polygonal complexes}\label{subsection: Cubulating groups acting on polygonal complexes}	
We now understand the structure of the hypergraph stabilisers and the separation in the wallspaces $(X,\mathcal{W})$.

For a metric polygonal complex $X$, let $D(X)$ be the maximal circumference of a polygonal face in $X$. We will be considering $G$ acting properly discontinuously and cocompactly on a polygonal complex $X$ so that $D(X)= D(G\backslash X)$ is finite. 

	\begin{lemma}\label{lem: edge hypergraph separation}
		Let $X$ be a simply connected $CAT(0)$ polygonal complex with $G\backslash X$ gluably edge $\pi$-separated. Let $\gamma$ be a finite geodesic in $X$ of length at least $4D(X)$. There exists an edge hypergraph $\Lambda$ that separates the endpoints of any finite geodesic extension of $\gamma$.
	\end{lemma}
	\begin{proof}
		Since $\gamma$ is of length at least $4D(X)$, we can find a subgeodesic $\delta$ of $\gamma$ of length at least $2D(X)$ that starts at a point $v\in X^{(1)}$ and ends at $w\in X^{(1)}$. If $\delta$ passes through the interior of a $2$-cell $f$ then, as $\delta$ is of length at least $2D(X)$, it meets the boundary $\partial f$ at two points $u_{1},u_{2}$. The sides of the polygonal faces are geodesic, and geodesics are unique in $CAT(0)$ spaces, so that there must exist a vertex $w$ in $\partial f$ lying between $u_{1}$ and $u_{2}$.
		
		Choose a cutset $C$ in $Lk(w)$ containing $f$, and let $P$ be a chosen partition of $\pi_{0}(Lk(w)-C)$ so that the endpoints of $f$ lie in different elements of $P$ (this must exist by assumption). Let $\Lambda$ be any hypergraph passing through $(C,P)$ in $Lk(w)$: by Lemma \ref{lem: separation condition}, $\Lambda$ separates the endpoints of the subpath of $\delta$ between $u_{1}$ and $u_{2}$: as geodesics in $X$ are unique, it follows that $\Lambda$ intersects any geodesic extension of $\delta$ exactly once, and so separates the endpoints of any geodesic extension of $\delta$. Otherwise $\delta$ lies strictly in $X^{(1)}$: $\delta$ is of length at least $2D(X)$ and so it must intersect at least two primary vertices. Therefore $\delta$ contains an edge of the form  $[u_{1},u_{2}]$ for some primary vertices $u_{1},u_{2}$: this edge must be geodesic. Furthermore, the geodesic $[u_{1},u_{2}]$ contains a secondary vertex $s$. Let $P$ be a partition  of $Lk(s)-E(s)$ so that the endpoints of $[u_{1},u_{2}]$ lie in different elements of $P$ (this must exist by assumption). Let $\Lambda$ be the hypergraph passing through $(C,P)$ in $Lk(s)$, it follows by Lemma \ref{lem: separation condition} that $\Lambda$ separates the endpoints of any finite geodesic extension of $\delta$.
	\end{proof}
	
Similarly, we have the following.
	\begin{lemma}\label{lem: vertex hypergraph separation}
		Let $X$ be a simply connected $CAT(0)$ polygonal complex with $G\backslash X$ gluably vertex $\pi$-separated. Let $\gamma$ be a finite geodesic in $X$ of length at least $4D(X)$. There exists a vertex hypergraph $\Lambda$ that separates the endpoints of any finite geodesic extension of $\gamma$.
	\end{lemma}
	\begin{proof}
		Again, since $\gamma$ is of length at least $4D(X)$, we can write $\gamma=\gamma_{1}\cdot\delta\cdot\gamma_{2}$, where each $\gamma_{i}$ is of length at least $D(X)\slash 2$, and $\delta$ is a path of length between $D(X)$ and $2D(X)$ that starts at a point $v\in X^{(1)}$ and ends at $w\in X^{(1)}$. 
		
	First suppose that  $\delta$ contains a nontrivial subpath, $\delta'$, which contains exactly one point of $X^{(1)}$, $u$, in its interior. Let $e$ be the edge of $X$ containing $u$. Since $\delta'$ is geodesic, we see that $i(\delta')$ and $t(\delta')$ lie in two distinct faces $F,F'$, both adjacent to $e$. In $Lk(i(e))$, $F,F'$ are two edges adjacent to $e$, and so, as $G\backslash X$ is gluably $\pi$-separated. there exists a $\pi$-separated cutset $C\ni e$ and partition $P$ of $\pi_{0}(Lk(i(e))-C)$ with $F,F'$ lying in distinct elements of $P$. 
	
	Let $\Lambda$ be any hypergraph passing through $(C,P)$ in $Lk(i(e))$. By Lemma \ref{lem: separation condition} $\Lambda$ separates the endpoints of $\delta'$, and hence the endpoints of $\gamma$.
		
		Otherwise, $\delta$ is contained in $X^{(1)}$: as we have not subdivided $X$,  $v$ is a primary vertex of $X$. Let $\delta_{1}$, $\delta_{2}$ be the two subpaths of $\gamma$ incident to $v$: as $\gamma$ is geodesic, $d_{Lk(v)}(\delta_{1},\delta_{2})\geq \pi$. Let $C$ be the vertex cutset such that $\gamma_{1}$ and $\gamma_{2}$ lie in different components of $Lk(v)-C$ and let $P$ be a chosen partition of $\pi_{0}(Lk(v)-C)$ separating $\gamma_{1}$ and $\gamma_{2}$ (this exists as $G\backslash X$ is gluably vertex $\pi$-separated). Let $\Lambda$ be any vertex hypergraph passing through $(C,P)$ in $Lk(v)$: by Lemma \ref{lem: separation condition} this separates $\gamma_{1}$ and $\gamma_{2}$, and so separates the endpoints of $\gamma$.
	\end{proof}

 We now turn our attention to finding codimension-$1$ subgroups. We first note the following lemma concerning $CAT(0)$ geometry.
		
	\begin{lemma}\label{lem: geodesic rays diverge}
	Let $Y$ be a $CAT(0)$ space, and let $\gamma_{1},\gamma_{2}$ be infinite one-ended geodesics starting from the same point. If there exists $r>0$ such that $\gamma_{1}\subseteq \mathcal{N}_{r}(\gamma_{2})$, then $\gamma_{1}=\gamma_{2}.$
	\end{lemma}
	\begin{proof}
	Let $p$ be the common start point of $\gamma_{1},\gamma_{2}$ and let $\theta=\angle_{p}(\gamma_{1},\gamma_{2}).$ Since $\gamma_{1}\subseteq \mathcal{N}_{r}(\gamma_{2})$, for all $t>0$ there exists $t'(t)>0$ such that $d(\gamma_{1}(t),\gamma_{2}(t'))\leq r$. However, $d(\gamma_{1}(t),p)\rightarrow \infty $ as $t\rightarrow\infty$, so that $d(\gamma_{2}(t'(t)),p)\rightarrow \infty $ as $t\rightarrow\infty$. Consider the Euclidean comparison triangle for the geodesics $\gamma_{1}(t)$ and $\gamma_{2}(t'(t))$: this has third side length at most $r$, and so has angle at $p$ of $\theta(t)\rightarrow 0$ as $t\rightarrow \infty.$ However, $\theta\leq \theta(t)$ for all $t$, and so $\theta=0$. It follows that $\gamma_{1}=\gamma_{2}$ in a closed neighbourhood of $p$, and so the set $\{t\;:\;\gamma_{1}(t)=\gamma_{2}(t)\}$ is clopen. The result follows.
	\end{proof}
Using this we can prove that hypergraph stabilizers have subgroups that are codimension-$1$ in $G$. Let $G$ be a group with finite generating set $S$ and let $\Gamma$ be the Cayley graph of $G$ with respect to $S$. A subgroup $H$ of $G$ is \emph{codimension-$1$} if the graph $H\backslash \Gamma$ has at least two ends, i.e. for some compact set $K$, $H\backslash\Gamma - K$ contains at least two infinite components. 
	\begin{lemma}\label{lem: hypergraph stabilizers are codimension-1}
	Let $G$ be a group acting properly discontinuously and cocompactly on a simply connected $CAT(0)$ polygonal complex $X$ such that $G\backslash X$ is (weakly) gluably $\pi$-separated. Let $\Lambda$ be a hypergraph in $X$. For any component $U_{\Lambda}$ of $X-\Lambda$, the group $$H_{U}=Stab_{Stab(\Lambda)}(U_{\Lambda})\cap Stab_{Stab(\Lambda)}(X-\overline{U_{\Lambda}})$$ is virtually free, and is quasi-isometrically embedded and codimension-$1$ in $G$.
	\end{lemma}
This again follows by \cite[Theorem 2.9]{Hruska-Wise}: we provide a direct proof for completeness.
	\begin{proof}
	We prove this in the case that $\Lambda$ is an edge hypergraph: the case for vertex hypergraphs is identical. By Lemma \ref{lem: edge hypergraphs are leafless convex}, $\Lambda$ is a convex tree. Since $\partial U_{\Lambda}\subseteq \Lambda$, $\partial U_{\Lambda}$ is a convex tree. We see that $H_{U}$ is of index at most $2$ in $Stab_{Stab(\Lambda)}(U_{\Lambda})$. By Lemma \ref{lem: wall stabilizers are cocompact}, $H_{U}$ acts properly discontinuously and cocompactly on $\partial U_{\Lambda}$: it follows that $H_{U}$ is virtually free and quasi-isometrically embedded in $G$. Furthermore, by Lemma \ref{lem:  separation condition} $X- \Lambda$ consists of at least two path-connected components, $\{U^{i}_{\Lambda}\}$. Let $V_{\Lambda}=X-\overline{U_{\Lambda}}$.
		
		Let $e_{1}$ and $e_{2}$ be vertices that lie in distinct components of $Lk(v)-C$ such that, in $X$, $e_{1}$ is an edge lying in $U_{\Lambda}\cup v$ and $e_{2}$ an edge lying in $V_{\Lambda}\cup v$. We construct two geodesics $\gamma_{1}$ and $\gamma_{2}$: let the first edge of $\gamma_{1}$ be $e_{1}$, and let $w$ be the endpoint of $e_{1}$ distinct from $v$. Since the links of vertices have no vertices of degree $1$ and have girth at least $2\pi$, it follows that there exists a vertex or edge, $a_{1}$, in $\Gamma=Lk(w)$ so that $d_{\Gamma}(e _{1}, m(a_{1}))\geq \pi$, and so we can extend $e_{1}$  to a geodesic $[v,m(a_{1})]$. We can continue in this fashion to construct a one-ended geodesic $\gamma_{1}$ that, by Lemma \ref{lem: separation condition}, lies in $U_{\Lambda}\cup v$ and (as geodesics are unique in $X$) intersects $\Lambda$ exactly once. Construct the geodesic $\gamma_{2}$ similarly, with first edge $e_{2}$ so that $\gamma_{2}$ intersects $\Lambda$ exactly once and lies in $V_{\Lambda}\cup v$. 
		
		By Lemma \ref{lem: geodesic rays diverge}, it follows that for any $r>0$, $\gamma_{1},\gamma_{2}\not\subseteq \mathcal{N}_{r}(\partial U_{\Lambda})$. Therefore $H_{U}\backslash X-H_{U}\backslash \partial U_{\Lambda}$ consists of at least two infinite components: $H_{U}\backslash U_{\Lambda}$ and $H_{U}\backslash V_{\Lambda}$. As $G$ is quasi-isometric to $X$, and $H_{U}$ is quasi-isometric to $\partial U_{\Lambda}$, the result follows.
	\end{proof}

	We will use Hruska--Wise's \cite{Hruska-Wise} generalisation of Sageev's construction of a $CAT(0)$ cube complex dual to a collection of codimension-$1$ subgroups, as introduced in \cite{Sageev-95}. We will only describe the $1$-skeleton of this cube complex. 
	
\begin{definition}[\emph{Orientation}]
 Let $(X,\mathcal{W})$ be a wallspace and $W=\{U,V\}$ a wall. An \emph{orientation} of $W$ is a choice $c(W)=(\overleftarrow{c(W)},\overrightarrow{c(W)})$ of ordering of the pair $W$. An \emph{orientation} of $\mathcal{W}$ is an orientation of each wall $W$ in $\mathcal{W}$.
\end{definition}
	
A $0$-cube in the dual cube complex $\mathcal{C}(X,\mathcal{W})$ corresponds to a choice of orientation $c$ of $\mathcal{W}$ such that that for any element $x\in X$, $x$ lies in $\overleftarrow{c(W)}$ for all but finitely many $W\in\mathcal{W}$, and $\overleftarrow{c(W)}\cap\overleftarrow{c(W')}\neq \emptyset$ for all $W,W'\in\mathcal{W}$. Two $0$-cells are joined by a $1$-cell if there exists a unique wall to which they assign opposite orientations.

Sageev analysed the properness and cocompactness of the group action on $\mathcal{C}(X,\mathcal{W})$ in \cite{Sageev-97}, and this was generalized by Hruska--Wise in \cite{Hruska-Wise}. We will use the following, as they are the easiest criteria to verify in our setting.

\begin{theorem*}\cite[Theorem 1.4]{Hruska-Wise}\label{prop: proper actions}
  Suppose $G$ acts on a wallspace $(X,\mathcal{W})$, and the action on the underlying metric space (X,d) is metrically proper. If there exists constants $\kappa,\epsilon>0$ such that for any $x,y\in X$, 
  $$\#(x,y)\geq \kappa d(x,y)-\epsilon,$$
  then $G$ acts metrically properly on $C(X, \mathcal{W})$.
\end{theorem*}
\begin{theorem*}\cite[Lemma 7.2]{Hruska-Wise}\label{thm: Sageev cubulation cocompactness}
Let $G$ act on a wallspace (X,W). Suppose there are finitely many orbits of collections of pairwise transverse walls in $X$. Then $G$ acts cocompactly on $C(X,\mathcal{W})$.
\end{theorem*}	

	This is sufficient to prove Theorem \ref{mainthm: cubulating groups}.
	
	\begin{proof}[Proof of Theorem \ref{mainthm: cubulating groups}]
	If $G\backslash X$ is gluably weakly $\pi$-separated, then by Lemma \ref{lem: hypergraph stabilizers are codimension-1}, $G$ contains a virtually free codimension-$1$ subgroup.
		
		Now suppose $G\backslash X$ is a gluably $\pi$-separated complex. $X$ is locally finite, and $G$ acts properly discontinuously on $X$, so acts metrically properly on $X$. Construct the hypergraph wallspace for $X$. Then by Lemmas \ref{lem: edge hypergraph separation} and \ref{lem: vertex hypergraph separation}, $$\#_{\Lambda}(p,q)\geq d_{X}(p,q)\slash 4D(X) -1.$$	
By Lemma \ref{lem: wallspace separation from hypergraph separation}, this implies that $$\#(p,q)\geq d_{X}(p,q)\slash 4D(X)-1:$$
by \cite[Theorem 1.4]{Hruska-Wise} it follows that $G$ acts properly discontinuously on the cube complex $C(X,\mathcal{W})$.

		Now suppose that $G$ is hyperbolic, so that $X$ is also hyperbolic. As hypergraphs are convex and hypergraph stabilisers are cocompact, by \cite{Gitik-Mitra-Rips-Sageev_1998widths} (c.f. \cite{Sageev-97}) there is an upper bound on the number of pairwise intersecting hypergraphs. For any point $x\in \Lambda$ there is a finite upper bound on the number of components of $X-\Lambda$ intersecting $x$, and so by Lemma \ref{lem: finitely many wall orbits}, we see there is an upper bound on the size of a collection of pairwise transverse walls. As $G$ acts cofinitely on the set of walls, it follows that the hypothesis of \cite[Lemma 7.2]{Haglund-Wise} are met, and so 
		$G$ acts cocompactly on the $CAT(0)$ cube complex $C(X,\mathcal{W})$:
	by \cite[Theorem 1.1]{Agol13}, we conclude that $G$ is virtually special. 
	\end{proof}

%----------------------------------------------------------------------------------------
%   Finding separated cutsets by computer search
%---------------------------------------------------------------------------------------
\section{Finding separated cutsets by computer search}\label{section: finding cutsets}
	In this short section, we discuss how to find separated cutsets by computer search. Let $\Gamma$ be a finite metric graph, and let $I(\Gamma)=V(\Gamma)$ or $E(\Gamma)$. 
	Define $d_{I}(x,y)=d_{\Gamma}(x,y)$ if $x,y\in V(\Gamma)$ and $d_{I}(x,y)=d_{\Gamma}(m(x),m(y))$ if $x,y\in E(\Gamma)$.
	Let $\sigma>0$. The $\sigma$-separated cutsets of $\Gamma$ that lie in $I(\Gamma)$ can be found in the following way: we can define a dual graph $\bar{\Gamma}$ by $V(\bar{\Gamma})=I(\Gamma),$ and 
		$$E(\bar{\Gamma})=\{(x,y)\in I(\Gamma)^{2}\;:\;x\neq y\;\mbox{and}\;d_{I}(x,y)< \sigma \}.$$
		Finding $\sigma$-separated cut sets in $\Gamma$ then corresponds to finding independent vertex sets in $\bar{\Gamma}$ and checking if they are cut sets in $\Gamma$. Importantly, finding independent vertex sets can be done relatively efficiently.
		
		See \href{ https://github.com/CJAshcroft/Graph-Cut-Set-Finder}{\color{blue} https://github.com/CJAshcroft/Graph-Cut-Set-Finder} for the implementation of the above algorithm, and for the code used to find cutsets in the following sections.

%----------------------------------------------------------------------------------------
	%   Application to triangular buildings
	%----------------------------------------------------------------------------------------
	\section{Triangular buildings}\label{section: generalized quadrangles}
	
		In the following section, we prove Corollary \ref{mainthm: cubulating the generalized quadrangle}.
		In \cite{Kangaslampi-Vdovina} and \cite{Carbone-Kangaslampli-Vdovina_2012} all groups acting simply transitively on triangular buildings whose links are the minimal generalized quadrangle (see Figure \ref{fig: generalized quadrangle}) were classified. We apply Theorem \ref{mainthm: cubulating groups} to these groups, proving they are virtually special by considering the separation of the minimal generalized quadrangle.

	\begin{figure}[ht]
		\centering
	\includegraphics[scale=0.75]{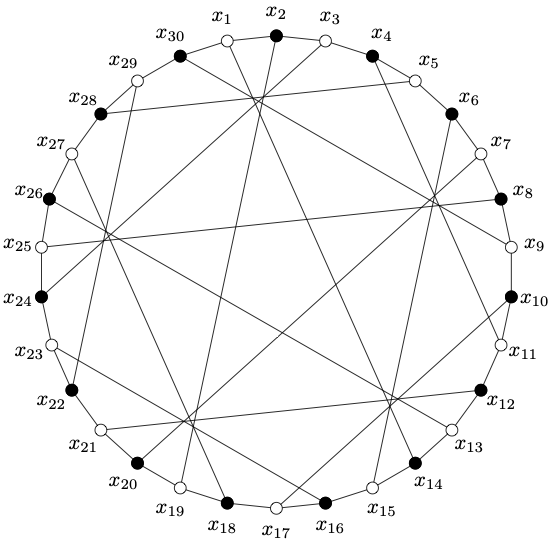}
		\caption{The minimal generalized quadrangle.}\label{fig: generalized quadrangle}
	\end{figure}

	\begin{table}[H]
    \begin{tabular}{c|ccl}
  $x_{i}$  &\multicolumn{3}{c}{$x_{j}$ adjacent to $x_{i}$}\\
    \hline

1&2& 14& 30\\
2&1& 3& 19\\
3&2& 4& 24\\
4&3& 5& 11\\
5&4& 6& 28\\
6&5& 7& 15\\
7&6& 8& 20\\
8&7& 9& 25\\
9&8& 10& 30\\
10&9& 11& 17
\end{tabular}
    \begin{tabular}{c|ccl}
  $x_{i}$  &\multicolumn{3}{c}{$x_{j}$ adjacent to $x_{i}$}\\
    \hline
11&4&10& 12\\
12&11& 13& 21\\
13&12& 14& 26\\
14&1& 13& 15\\
15&6& 14& 16\\
16&15& 17& 23\\
17&10& 16& 18\\
18&17& 19& 27\\
19&2& 18& 20\\
20&7& 19& 21
\end{tabular}
    \begin{tabular}{c|ccl}
  $x_{i}$  &\multicolumn{3}{c}{$x_{j}$ adjacent to $x_{i}$}\\
    \hline

21&12& 20& 22\\
22&21& 23& 29\\
23&16& 22& 24\\
24&3& 23& 25\\
25&8& 24& 26\\
26&13&25&27\\
27&18&26&28\\
28&5&27&29\\
29&22&28&30\\
30&1&9&29
\end{tabular}
    \caption{Edge incidences for the minimal generalized quadrangle}

\end{table}
	\begin{lemma}\label{lem: separation of the generalized quadrangle}
		Let $\Gamma$ be the minimal generalized quadrangle equipped with the combinatorial metric. Then $\Gamma$ is weighted (strongly) edge $3$-separated.  
	\end{lemma}
	
	\begin{proof}
	By a computer search, we find the following exhaustive list of $3$-separated edge cut sets in $\Gamma$:
		\begin{enumerate}[label=\hspace*{-10pt}]
	{
			\item\hspace*{-10pt}$C_{1}=\{\fe{1}{2},\fe{4}{5},\fe{7}{20},\fe{9}{10},\fe{12}{13},\fe{15}{16},\fe{18}{27},\\\hspace*{25pt}\fe{22}{29},\fe{24}{25}\},$
			\item\hspace*{-10pt}$C_{2}=\{\fe{1}{2},\fe{4}{11},\fe{6}{15},\fe{8}{9},\fe{13}{26},\fe{17}{18},\fe{20}{21},\\\hspace*{25pt}\fe{23}{24},\fe{28}{29}\},$
			\item\hspace*{-10pt}$C_{3}=\{\fe{1}{14},\fe{3}{4},\fe{6}{7},\fe{9}{10},\fe{12}{21},\fe{16}{23},\fe{18}{19},\\\hspace*{25pt}\fe{25}{26},\fe{28}{29}\},$
			\item\hspace*{-10pt}$C_{4}=\{\fe{1}{14},\fe{3}{24},\fe{6}{7},\fe{9}{10},\fe{12}{21},\fe{16}{23},\fe{18}{19},\\\hspace*{25pt}\fe{25}{26},\fe{28}{29}\},$
			\item\hspace*{-10pt}$C_{5}=\{\fe{1}{30},\fe{3}{4},\fe{6}{15},\fe{8}{25},\fe{10}{17},\fe{12}{13},\fe{19}{20},\\\hspace*{25pt}\fe{22}{23},\fe{27}{28}\},$
			\item\hspace*{-10pt}$C_{6}=\{\fe{1}{30},\fe{3}{24},\fe{5}{28},\fe{7}{8},\fe{10}{11},\fe{13}{26},\fe{15}{16},\\\hspace*{25pt}\fe{18}{19},\fe{21}{22}\},$
			\item\hspace*{-10pt}$C_{7}=\{\fe{2}{3},\fe{5}{6},\fe{8}{25},\fe{10}{11},\fe{13}{14},\fe{16}{23},\fe{18}{27},\\\hspace*{25pt}\fe{20}{21},\fe{29}{30}\},$
			\item\hspace*{-10pt}$C_{8}=\{\fe{2}{3},\fe{5}{28},\fe{7}{20},\fe{9}{30},\fe{11}{12},\fe{14}{15},\fe{17}{18},\\\hspace*{25pt}\fe{22}{23},\fe{25}{26}\},$
			\item\hspace*{-10pt}$C_{9}=\{\fe{2}{19},\fe{4}{5},\fe{7}{8},\fe{10}{17},\fe{12}{21},\fe{14}{15},\fe{23}{24},\\\hspace*{25pt}\fe{26}{27},\fe{29}{30}\},$
			\item\hspace*{-10pt}$C_{10}=\{\fe{2}{19},\fe{4}{11},\fe{6}{7},\fe{9}{30},\fe{13}{14},\fe{16}{17},\fe{21}{22},\\\hspace*{25pt}\fe{24}{25},\fe{27}{28}\}.$
			}
		\end{enumerate}
	
	$\Gamma$ is connected and contains no vertices of degree $1$. The cutsets sets are $3$-separated, and $\cup_{i}C_{i}=E(\Gamma)$. In fact, each cutset is minimal, and so is certainly proper. Furthermore, every edge appears in exactly two cutsets: assigning each cutset weight $1$ we see that the weight equations are satisfied, and so $\Gamma$ is weighted edge $3$-separated.
	
	In fact, by a computer search we can see that $\Gamma$ satisfies the conditions of Lemma \ref{lem: strong edge sep condition}, and so is weighted strongly edge $3$-separated.
	\end{proof}
	Note that $C_{i}\cap C_{j}$ is nonempty for all $i$ and $j$, so that we are not able to use \cite[Example $4.3$]{Hruska-Wise}. However, we can apply Theorem \ref{mainthm: cubulating groups} to prove groups acting properly discontinuously and cocompactly on triangular buildings with the minimal generalized quadrangle as links are virtually special.

	\begin{proof}[Proof of Corollary \ref{mainthm: cubulating the generalized quadrangle}]
	Let $X$ be a simply connected polygonal complex such that every face has at least $3$ sides, and the link of every vertex is isomorphic to the minimal generalized quadrangle, $\Gamma$, and let $G$ be a group acting properly discontinuously and cocompactly on $X$. Since $\Gamma$ has girth $8$, $X$ can be endowed with a $CAT(-1)$ metric, so that $G$ is hyperbolic. Endow $X$ with the metric that makes each $k$-gonal face a regular unit Euclidean $k$-gon, so that $X$ is regular and the length of each edge in the link of a vertex is at least $\pi\slash 3$. As $\Gamma$ is weighted edge $3$-separated with the combinatorial metric, it follows that the links of $X$ are weighted edge $\pi$-separated. Hence by Lemma \ref{lem: solving gluing equations for minimal edge cutsets}, $G\backslash X$ is gluably $\pi$-separated. Furthermore, by Gromov's link condition, $X$ is $CAT(0)$.  Therefore, $G$ is hyperbolic and acts properly discontinuously and cocompactly on a simply connected  $CAT(0)$ triangular complex $X$ with $G\backslash X$ gluably $\pi$-separated, so acts properly discontinuously and cocompactly on a $CAT(0)$ cube complex by Theorem \ref{mainthm: cubulating groups}, and hence is virtually special by \cite[Theorem $1.1$]{Agol13}.
	\end{proof}

	%----------------------------------------------------------------------------------------
	%   Application to generalized ordinary triangular groups
	%----------------------------------------------------------------------------------------
	\section{Application to generalized triangular groups}\label{sec: gen triangles}
In this section we prove Theorem \ref{mainthm: cubulating generalized triangle groups} in Section \ref{subsection: Codimension-$1$ subgroups of generalized ordinary triangle groups}, Corollary \ref{coralph: small girth generalized triangle groups} in Section \ref{subsection: small girth generalized triangle groups}, and Corollary \ref{mainthm: cubulating dehn fillings of generalized triangle groups} in Section \ref{subsection: cubulating dehn fillings of generalized triangle groups}.	
%----------------------------------------------------------------------------------------
%   SUBSECTION: Cubulating TRIANGLE GROUPS
%----------------------------------------------------------------------------------------
		\subsection{Cubulating generalized ordinary triangle groups}\label{subsection: Codimension-$1$ subgroups of generalized ordinary triangle groups}
		We now consider generalized ordinary triangle groups, constructed in \cite{Lubotzky-Manning-Wilton} to answer a question of Agol and Wise: note that the case of $k=2$ corresponds to classical ordinary triangle groups.
		
		The first complex of groups we define uses the notation from \cite{Caprace-Conder-Kaluba-Witzel_triangle} to more easily align with their work. See e.g. \cite{Bridson-Haefliger} for further discussion of complexes of groups.
\begin{definition}[Generalized triangle groups]\label{def: generalized triangle}
        Consider the following complex of groups over $\mathcal{T}$, the poset of all subsets of $\{1,2,3\}$. Let $X_{1},X_{2},X_{3}$ be the vertex groups, and $A_{1},A_{2},A_{3}$ the edge groups, with the face group trivial, and homomorphisms $\phi_{i,i+1}:A_{i}\rightarrow X_{i+1},\;\phi_{i,i-1}:A_{i}\rightarrow A_{i-1}$ for $i=1,2,3$ taken $\mod 3$. Now, consider the coset graph $$\Gamma_{X_{i}}(\phi_{i-1,i}(A_{i-1}),\phi_{i+1,i}A_{i+1}).$$
Fix $k\geq 2$ and let each $A_{i}=\mathbb{Z}\slash k$. For graphs $\Gamma_{i}$, let $$\{D_{k}^{j}(\Gamma_{1},\Gamma_{2},\Gamma_{3})\}_{j}$$ be the family of complexes of groups obtained by choosing $X_{i}$ and $\phi_{i,i\pm 1}$ such that for each $i$ $$\Gamma_{X_{i}}(\phi_{i-1,i}(A_{i-1}),\phi_{i+1,i}A_{i+1})\cong \Gamma_{i}.$$ A group $$G^{j}_{k}(\Gamma_{1},\Gamma_{2},\Gamma_{3})=\pi_{1}(D^{j}_{k}(\Gamma_{1},\Gamma_{2},\Gamma_{3}))$$ is called a \emph{($k$-fold) generalized triangle group}.
\end{definition}	
Bridson and Haefliger considered the developability of a complex of groups in \cite[III.$\mathcal{C}$]{Bridson-Haefliger}. The following is well known: see e.g. \cite[Theorem 3.1]{Caprace-Conder-Kaluba-Witzel_triangle}.
\begin{prop}\label{lem: developability of gen triangle}
Suppose that $girth(\Gamma_{i})\geq 6$ for each $i$. Then $G_{k}^{j}(\Gamma_{1},\Gamma_{2},\Gamma_{3})$ acts properly and cocompactly on a triangular complex $X^{j}(\Gamma_{1},\Gamma_{2},\Gamma_{3})$ such that the link of each vertex is isomorphic to $\Gamma\in\{\Gamma_{i}\}_{i}$. If $girth(\Gamma_{1})>6$, then $G$ is hyperbolic.
\end{prop}
	\begin{definition}[Generalized ordinary triangle groups]\label{def: generalized triangle 2}
		Consider the following complex of groups. Fix $k\geq 2$, and identify the boundaries of $k$ $2$-simplices to construct a simplicial complex $\mathcal{K}$ with three vertices $v_{1},v_{2},v_{3}$, three edges $e_{1}, e_{2}, e_{3}$, and $k$ $2$-simplices. Then $Lk(v_{i})\simeq C_{k,2}$, the cage graph on $k$ edges, i.e. the smallest $k$ regular graph of girth $2$. 
		
		Let $P_{i}=\pi_{1}(Lk(v_{i}))$, and let $G_{0,k}$ be the free group on $2k-2$ letters. Note that we can view $G_{0,k}$ as the fundamental group of a complex of groups with underlying simplicial complex $\mathcal{K}$ and vertex groups $P_{i}$. Now, let $\Gamma_{i}\looparrowright Lk(v_{i})$ be finite-sheeted normal covering graphs, with associated normal subgroups $Q_{i}\unlhd P_{i}$. Let $D$ be a complex of groups with underlying complex $\mathcal{K}$ and (finite) vertex groups $V_{i}=P_{i}\slash Q_{i}$. Since there are choices for the above complex, we will let $D^{j}_{0,k}(\Gamma_{1},\Gamma_{2},\Gamma_{3})$, $j=1,\hdots ,$ be the finite exhaustive list of possible complexes of groups achieved by the above construction.
		
		Form the \emph{($k$-fold) generalized ordinary triangular group} $$ G^{j}_{0,k}(\Gamma_{1},\Gamma_{2},\Gamma_{3})=\pi_{1}(D^{j}_{0,k}(\Gamma_{1},\Gamma_{2},\Gamma_{3}))=G_{0,k}\slash \langle\langle Q_{1}\cup Q_{2}\cup Q_{3}\rangle\rangle .$$
	\end{definition}
	Note that in this definition the graphs $\Gamma_{i}$ are covers of $C_{k,2}$ so that they are connected, contain no cut edges, and have girth at least $2$. 
	
	Theorem \ref{mainthm: cubulating groups}, along with Proposition \ref{lem: developability of gen triangle}, and \cite[Proposition 3.2]{Lubotzky-Manning-Wilton} below, allow us to cubulate $G^{j}_{k}(\Gamma_{1},\Gamma_{2},\Gamma_{3})$ and $G^{j}_{0,k}(\Gamma_{1},\Gamma_{2},\Gamma_{3})$ when given enough information about each of $\Gamma_{1},\Gamma_{2},\Gamma_{3}$. The purpose of this subsection is to provide a way to prove such a group acts properly discontinuously on a $CAT(0)$ cube complex by considering $\Gamma_{1}$ alone. Again, see Section \ref{subsec: Link conditions} for the relevant definitions.
	
	\begin{theorem}\label{mainthm: cubulating generalized triangle groups}
	Let $\Gamma_{i}\looparrowright C_{k,2}$ be finite-sheeted covers such that $girth(\Gamma_{i})\geq 6$ for each $i$, and let $G=G^{j}_{0,k}(\Gamma_{1},\Gamma_{2},\Gamma_{3})$ or $G=G^{j}_{k}(\Gamma_{1},\Gamma_{2},\Gamma_{3})$.
If $\Gamma_{1}$ is weighted strongly edge $3$-separated, then $G$ acts properly discontinuously on a $CAT(0)$ cube complex. If, in addition, $G$ is hyperbolic, then this action is cocompact. 
	\end{theorem}

	We use the following proposition, as stated in \cite[Proposition 3.2]{Lubotzky-Manning-Wilton}, which follows by an application of \cite[Theorem III.$\mathcal{C}$.4.17]{Bridson-Haefliger}.
	
	\begin{proposition*} \cite[Proposition 3.2]{Lubotzky-Manning-Wilton}
		If $girth(\Gamma_{i})\geq 6$ for each $i$, then $G^{j}_{0,k}(\Gamma_{1},\Gamma_{2},\Gamma_{3})$ acts properly discontinuously and cocompactly on a simply connected simplicial complex $X^{j}(\Gamma_{1},\Gamma_{2},\Gamma_{3})$ with links isomorphic to $\Gamma$, where $\Gamma\in\{\Gamma_{1},\Gamma_{2},\Gamma_{3}\}.$ Furthermore, if $girth(\Gamma_{1})\geq 6$ for each $i$ and $girth(\Gamma_{1})>6$, then $G_{0,k}(\Gamma_{1},\Gamma_{2},\Gamma_{3})$ is hyperbolic.
	\end{proposition*}

Now, fix $j$, let $G=G_{k}^{j}(\Gamma_{1},\Gamma_{2},\Gamma_{3})$ or $G=G_{0,k}^{j}(\Gamma_{1},\Gamma_{2},\Gamma_{3})$, and let $X=X^{j}(\Gamma_{1},\Gamma_{2},\Gamma_{3})$ be as above. Note that the antipodal graph $\Delta_{G\backslash X}$ is the disjoint union of three components $\Delta_{1},\Delta_{2},\Delta_{3}$, such that for any vertex $v\in \Delta_{i}$ either $v$ is secondary, or $Lk_{G\backslash X}(v)\cong \Gamma_{i}$. Suppose $\Gamma_{1}$ is a weighted strongly edge $3$-separated graph, and endow $X$ with the metric that turns each triangle into a unit equilateral Euclidean triangle: $X$ is $CAT(0)$ with this metric. Then for each $v\in V(\Delta_{1})$, $Lk_{G\backslash X}(v)$ is a strongly edge $\pi$-separated graph. Since cutsets are proper, we can assign to every cutset the canonical partition: as discussed in Section \ref{subsection: Examples of solutions of the gluing equations} this is sufficient for cubulation, and therefore we omit the reference to partitions for the remainder of this section. As in Section \ref{subsection: Constructing hypergraphs in polygonal complexes} construct the graphs $\underline{\Lambda}_{1},\hdots ,\underline{\Lambda}_{m}$ as images of
$$\Sigma_{i}\looparrowright\Delta_{1}\looparrowright G\backslash X.$$

In particular, if a vertex $v$ has $Lk_{X}(v)=\Gamma_{1}$, we have a hypergraph passing through every $\pi$-separated edge cutset in $Lk_{X}(v)=\Gamma_{1}$. As in Section \ref{subsection: Hypergraph stabolizers and wallspaces}, we can again build the system of hypergraph walls.

We now analyse the separation of this complex by hypergraphs.
\begin{lemma}\label{lem:  strongly separated triangle}
Suppose that $\Gamma_{1}$ is weighted strongly edge $3$-separated, and let $\gamma$ be a geodesic in $X$ of length at least $100$. There exists a hypergraph $\Lambda$ such that $\Lambda$ separates the endpoints of any finite geodesic extension of $\gamma$.
\end{lemma}
\begin{proof}
Since $\gamma$ has length at least $100$, we may write $\gamma=\beta \cdot\gamma_{1}\cdot\gamma_{2}\cdot\gamma_{3}\cdot \delta$ such that 
$1\leq l(\gamma_{i})\leq \sqrt{3\slash 2}$, $l(\beta),l(\delta)\geq 40$ and the endpoints of each $\gamma_{i}$ lie in $X^{(1)}$. We can see that either:
\begin{enumerate}[label = \bf{case} $\boldsymbol{\alph*})$]
 
    \item $\gamma_{2}$ contains an edge of $X^{(1)}$ of the form $[u,v]$, 
    \item or $\gamma_{2}$ contains a subpath that intersects $X^{(1)}$ at exactly two points $x,y$ in $\partial T$ for some $2$-cell $T$.
\end{enumerate}

Now consider case $a)$. There are two subcases to consider.
\begin{enumerate}[label = \bf{Case} $\boldsymbol{a.\roman*}$)]
\item $u$ and $v$ have links isomorphic to $\Gamma_{2}$ and $\Gamma_{3}$ respectively (or vice versa),
\item or $v$ has link isomorphic to $\Gamma_{1}.$
\end{enumerate}
In case $a.i)$, $\gamma_{2}$ contains a secondary vertex $x$ that is opposite to some $w$ with $Lk_{X}(w)\cong\Gamma_{1}$. By Lemma \ref{lem: separation condition}, the hypergraph passing through $x$ and $w$ therefore separates the endpoints of $\gamma'$, and so the endpoints of any geodesic extension of $\gamma.$

In case $a.ii)$, consider the path $\gamma_{3}.$ If $\gamma_{3}$ is not an edge then $\gamma_{3}$ satisfies the hypothesis of case $b)$ using $\gamma_{3}$ in place of $\gamma_{2}$. Otherwise we may assume that $\gamma_{2}\cdot \gamma_{3}=[u,v]\cdot[v,w]$. Now, $d_{Lk(v)}([u,v],[v,w])\geq \pi$ as $\gamma_{2}\cdot \gamma_{3}$ is geodesic: let $C$ be the cutset separating $[u,v]$ and $[v,w]$ in $\Gamma_{1}$ (this exists as $\Gamma_{1}$ is strongly $3$-separated): by Lemma \ref{lem: separation condition} the hypergraph passing through $C$ in $Lk(v)$ separates the endpoints of $\gamma.$

\begin{figure}[H]
\begin{minipage}[b]{0.4\textwidth}
	 \includegraphics[scale=0.9]{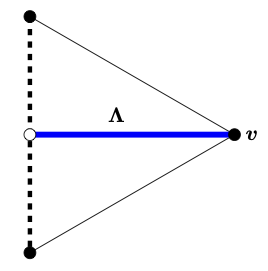}\centering
	    \caption*{Case a.i)}
	    \end{minipage}
 \begin{minipage}[b]{0.4\textwidth}
 \includegraphics[scale=0.9]{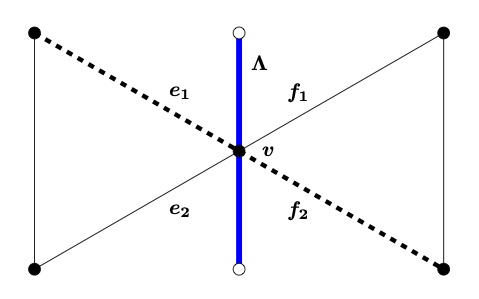}	  \centering
	    \caption*{\hspace*{70pt}Case a.ii)}
\end{minipage}
	    \label{fig:gentriangle1}
\end{figure}

For case $b)$ there are three subcases:
\begin{enumerate}[label= \bf{case} $\boldsymbol{b.\roman*)}$]
    \item the two paths in $\partial T$ from $x$ to $y$ each contain one of the vertices $u,v$ such that $u$ is primary with $Lk(u)\cong \Gamma_{1}$, and $v$ is secondary and antipodal to $u$ in $\partial T$,
    \item one of the two paths in $\partial T$ from $x$ to $y$ contains both of the vertices $u,v$ where $u$ is primary with $Lk(u)\cong\Gamma_{1}$, and $v$ is secondary and opposite to $u$,
    \item or $\gamma_{2}=[u,v]$ where $u$ is secondary and opposite to $v$, with $Lk(v)\cong\Gamma_{1}$.
\end{enumerate}
In case b.i), by Lemma \ref{lem: separation condition} the hypergraph passing through $u$ and $v$ separates the endpoints of $\gamma_{2}$ and so the endpoints of $\gamma$. 

Consider case $b.ii)$. Let $T_{x}$, $T_{y}$ be the two $2$-cells adjacent$T$ containing the vertex $x$ and $y$ respectively, with $\gamma$ passing through both $T_{x}$, $T_{y}$. Note that $x$ and $y$ lie on different edges of $\partial T$. Suppose that $\gamma_{1}$ passes through $x$ and $\gamma_{3}$ passes through $y$: we may see that by a simple Euclidean geometry argument for angles that either $\gamma_{1}$ or $\gamma_{3}$ satisfies case $b.i)$.

In case $b.iii)$ extend $\gamma_{2}$ through $v$ until we meet $X^{(1)}$ at a third point $w$: without loss of generality this can be written $\gamma_{1}\cdot \gamma_{2}=[u,v]\cdot [v,w]$. Now, as $\gamma$ is geodesic, we have $d_{Lk(v)}([u,v],[v,w])\geq 3$. Let $C$ be the cutset in $\Gamma_{1}$ such that  $[u,v]$ and $[v,w]$ are separated by $C$ (this exists as $\Gamma_{1}$ is strongly $3$-separated), and let $\Lambda$ be the hypergraph passing through $C$ in $Lk(v)$. By Lemma \ref{lem: separation condition} $\Lambda$ separates the endpoints of $\gamma$.

\begin{figure}[H]
\centering
\begin{minipage}[b]{0.25\textwidth}
\includegraphics[scale=0.9]{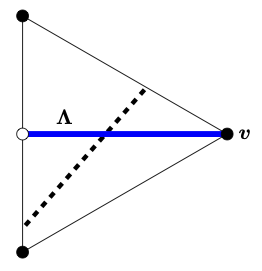}	    \centering
	    \caption*{Case b.i)}
\end{minipage}\hspace*{30pt}
\begin{minipage}[b]{0.49\textwidth}
\includegraphics[scale=0.8]{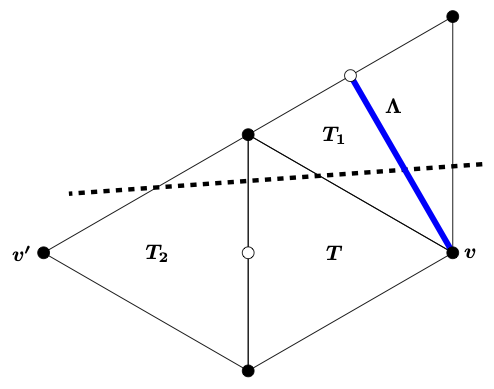}	   \centering
	    \caption*{\hspace*{70pt}Case b.ii)}
	    \end{minipage}
	    \label{fig:gentriangle2}
	\end{figure}
\begin{figure}[H]
\centering
\includegraphics[scale=0.9]{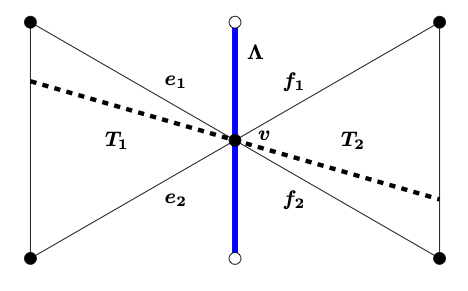}
\caption*{Case b.iii)}
\end{figure}	
\end{proof}

We can now prove Theorem \ref{mainthm: cubulating generalized triangle groups}.
\begin{proof}[Proof of Theorem \ref{mainthm: cubulating generalized triangle groups}]
By \cite[Proposition 3.2]{Lubotzky-Manning-Wilton}, the group $G$ acts properly discontinuously and cocompactly on a simply connected simplicial complex $X$. Endow this complex with the Euclidean metric: by Gromov's link condition X is $CAT(0)$ and has three types of vertices $\{v_{i}\}$ where $Lk(v_{i})=\Gamma_{i}$. 

If $\Gamma_{1}$ is strongly $3$-separated, then by Lemma \ref{lem:  strongly separated triangle}, we have $$\#(p,q)\geq d_{X}(p,q)\slash 100 -1.$$ The results then follow by \cite[Theorem 5.2]{Hruska-Wise} and \cite[Lemma 7.2]{Hruska-Wise} similarly to the proof of Theorem \ref{mainthm: cubulating groups}, using Lemma \ref{lem:  strongly separated triangle} in place of Lemma \ref{lem: edge hypergraph separation}.
\end{proof}

%%%%%%%%%%%%%%%%%%%%%%%%%%%%%%%%%%%%%%%%%%%%%%%%
\subsection{Small girth generalized triangle groups}\label{subsection: small girth generalized triangle groups}
To prove Corollary \ref{coralph: small girth generalized triangle groups}, we now analyse the separation of various small girth graphs considered in \cite{Caprace-Conder-Kaluba-Witzel_triangle}. These graphs arise in the work of \cite{Conder-Morton_1995classification,Conder-Dobcsanyi_2002trivalent,Conder-Malnic-Marusic-Potocnik_2006census} and are regular bipartite graphs with girth $6$ or $8$, diameter $3$ or $4$, and an edge regular subgroup of the automorphism group. Furthermore, they all have a vertex transitive automorphism group.

In particular, we have the following.

\begin{lemma*}\cite{Caprace-Conder-Kaluba-Witzel_triangle} Let $\Gamma$ be one of $\{F24A,F26A,F40A,F48A\}$. Then $Aut(\Gamma)$ acts vertex transitively. Let $\Gamma$ be one of $\{F24A,F26A,F40A,F48A, G54\}$: there exists a subgroup $H(\Gamma)\leq Aut(\Gamma)$ that acts freely and transitively on $E(\Gamma)$ and preserves the bipartition of $\Gamma.$

\end{lemma*}

We make the following definitions.
\begin{definition}[Cubic graphs]
      Let $\Gamma$ be a finite graph. It is \emph{cubic} if it is connected, bipartite, and trivalent.  
\end{definition}
\begin{definition}[$\dagger$-separated graphs]
Let $\Gamma$ be a graph. We say that $\Gamma$ is \emph{$\dagger$-separated} if:

\begin{enumerate}[label=$\roman*)$]
    \item $\Gamma$ is cubic,
    \item  $girth(\Gamma)= 6$ or $8$,
    \item and $\Gamma$ is disjointly weighted vertex $3$-separated by proper cutsets, (so that $\Gamma -C$ consists of exactly three components for each $C$).
    \end{enumerate}
\end{definition}
\begin{definition}[Good cubic graphs]
  A cubic graph is \emph{good} if $girth(\Gamma)=6$ or $8$, $diam(\Gamma)\leq 4$, $Aut(\Gamma)$ acts vertex transitively, and there exists a group $H(\Gamma)\leq Aut(\Gamma)$ that acts freely and transitively on $E(\Gamma)$ and preserves the bipartition of $\Gamma.$
\end{definition}
In the above definition, for any vertex $v$ of $\Gamma$, $H(\Gamma)_{v}$ is of order three and so cyclically permutes the neighbours of $v$.

Fix a vertex $v_{0}\in V(\Gamma).$ For each pair of vertices $v\neq w$, choose an element $\gamma_{v,w}\in Aut(\Gamma)$ with $\gamma_{v,w} v=w$, such that
\begin{enumerate}[label=$\roman*)$]
    \item $\gamma_{v,w}=\gamma_{v,v_{0}}\gamma_{v_{0},w},$
    \item $\gamma_{v,w}=\gamma_{w,v}^{-1},$
    \item and if $v,w\in V_{1}$ or $v,w\in V_{2}$, then $\gamma_{v,w}\in H(\Gamma)$.
\end{enumerate}
For each $v\in V(\Gamma)$ we will let neighbours of $v$ be defined as $w_{1}(v),w_{2}(v),w_{3}(v)$, so that $\gamma_{v_{0},v}w_{i}(v_{0})=w_{i}(v).$ We also assign to $H(\Gamma)_{v}$ a generator $h_{v}$ such that $h_{v}w_{1}(v)=w_{2}(v)$, $h_{v}w_{2}(v)=w_{3}(v)$, and so on, i.e. $h_{v}=\gamma_{v_{0},v}h_{v_{0}}\gamma_{v,v_{0}}$.

\begin{definition}[$\Large{*}$-separated cutsets]
Let $C$ be a vertex cutset in a graph $\Gamma$. We say $C$ is a \emph{$\Large{*}$-separated cutset} if $C$ is $3$-separated, for any vertex $w\in C$, there are two vertices $v,v'$ adjacent to $w$ such that $v$ and $v'$ lie in separate components of $\Gamma-C$, and $\Gamma - C$ contains exactly two components.

\end{definition}

\begin{definition}
   Let $\Gamma$ be a good cubical graph, and let $\mathcal{C}$ be a collection of $\Large{*}$-separated cutsets. For $v\in V(\Gamma)$, we define
   $$\large{*}(v,i,j)$$
   to be the set of all $\Large{*}$-separated cutsets $C\ni v$ such that $w_{i}(v)$ and $w_{j}(v)$ lie in the same connected component of $\Gamma - C$. We further define $$\mathcal{C}(v,i,j):=\mathcal{C}\cap \Large{*}(v,i,j).$$
\end{definition}

\begin{definition}[$\Large{*}$-separated graph]
Let $\Gamma$ be a graph. We say that $\Gamma$ is \emph{$\Large{*}$-separated} if:
\begin{enumerate}[label=$\roman*)$]
    \item $\Gamma$ is a cubic graph,
    \item $\Gamma$ is weighted vertex $3$-separated by a set $\mathcal{C}$ of $\Large{*}$-separated cutsets,
    \item for any  vertex $v$ and any $i\neq j$, $\mathcal{C}(v,i,j)$ is non-empty
    \item there exists an integer $M$ and positive integers $n(C)$ for each $C\in \mathcal{C}$ such that for any vertex $v$ and any $i\neq j $,
    $$\sum\limits_{C\in\mathcal{C}(v,i,j)}n(C)=\frac{M}{3}$$
\end{enumerate}
\end{definition}

\begin{definition}
        Let $v$ be any vertex in $\Gamma$. We define $D(v)=\{w\in \Gamma\;:\;d(v,w)\geq 5\}$.
\end{definition}
For ease, we prove the following lemma.

\begin{lemma}\label{lem: simple condition for 3 sep}
Let $\Gamma$ be a good cubic graph. Let $V_{1}\sqcup V_{2}$ be the bipartite partition of vertices, and choose $v_{1}\in V_{1}$. Suppose that there exists  $\Large{*}$-separated cutset $A_{i}\ni v_{1}$ such  that for each $u\in D(w_{1}(v_{1}))$, there exists some $i$, $w_{3}(v_{1})$ and $u$ lie in separate components of $\Gamma$ - $A_{i}$.
Then $\Gamma$ is $\Large{*}$-separated.
\end{lemma}
\begin{proof}
We need to show three separate things. Firstly we show that there exists a collection $\mathcal{C}$ of $\Large{*}$-separated cutsets so that $\Gamma$ is vertex $3$-separated by $\mathcal{C}$.
Recall the element $\gamma=\gamma_{v_{1},v_{2}}$, the element of $Aut(\Gamma)$ taking $v_{1}$ to $v_{2}:=w_{1}(v_{1})\in V_{2}$.

Let $H:=H(\Gamma)$ be the group acting edge-regularly on $\Gamma$ and preserving the bipartite partition. Let $A=\{A_{i}\}_{i}$, $B=\gamma \cdot A$, $\mathcal{A}=H\cdot A$, $\mathcal{B}=H\cdot B$, and $\mathcal{C}=\mathcal{A}\cup \mathcal{B}.$

By assumption, for some $i\neq j$ $\mathcal{C}(v_{1},i,j)$ is non-empty. For any vertex $v$, $\gamma_{v_{1},v}\mathcal{C}(v_{1},1,2)=\mathcal{C}(v,1,2)$, and furthermore,  $h_{v}\mathcal{C}(v,1,2)=\mathcal{C}(v,2,3)=h_{v}^{-1}\mathcal{C}(v,1,3)$. Therefore $\mathcal{C}(v,i,j)$ is non empty for all $v$ and all $i\neq j$. In particular for any vertex $v$ and $w,w'$ adjacent to $v$ there exists a cutset separating $w$ and $w'$.

 Now let $u,v$ be vertices distance at least $3$ apart. Note that $d(u,v)\leq 4$ as $diam(\Gamma)\leq 4$. Assume $d(u,v)=3$, and let $$p=(u,u_{1})(u_{1},u_{2})(u_{2},v)$$ be any edge path between $u$ and $v$. 

Now, suppose without loss of generality that $u=w_{1}(u_{1})$ and $u_{2}=w_{2}(u_{1})$. Then choosing a cutset $C\in\mathcal{C}(u_{1},1,3)$, $u$ and $u_{2}$ lie on separate components of $\Gamma-C$. Since $C$ is $3$-separated, and $u_{1}\in C$, it follows that $u$, $v$ are not elements of $C$. As $u$ is adjacent to $u_{1}$ and $v$ is adjacent to $u_{2}$, it follows that $u$ and $v$ lie in different components of $\Gamma-C$.

If $d(u,w)=4$, then we repeat the argument for 
$$p=(u,u_{1})(u_{1},u_{2})(u_{2},u_{3})(u_{3},v)$$ and for $C$ the cutset containing $u_{2}$ and separating $u_{1}$ and $u_{3}$. It now follows by Lemma \ref{lem: vertex separated condition} that $\Gamma$ is vertex $3$-separated.

If $d(u,w)=5,6$, consider the edge path 
$$p=(u,u_{1})(u_{1},u_{2})(u_{2},u_{3})(u_{3},u_{4})(u_{4},v),$$
or 
$$p=(u,u_{1})(u_{1},u_{2})(u_{2},u_{3})(u_{3},u_{4})(u_{4},u_{5})(u_{5},v).$$
We may map $u_{1}$ to $v_{1}$ and $u$ to $w_{3}(v_{1})$: by taking $A$ and mapping back, by assumption this cutset separates $u$ and $v$.

Finally we wish to find the positive integers $M$ and $n(C)$. This immediately implies the weight equations can be solved, and so $\Gamma$ is weighted vertex $3$-separated with respect to $\mathcal{C}.$ The proof is similar to the proof of Lemma \ref{lem: vertex transitive automorphism group can solve equations} concerning vertex transitive automorphism groups.

Let $\tilde{\mathcal{C}}=H\cdot A\cup H\cdot B$ counted \emph{with multiplicity}. Let $u,v\in V_{1}$.  For $i=1,2,3,$ we have $\gamma_{u,v}(w_{i}(u))=w_{i}(v)$. It follows that for $C\in \tilde{\mathcal{C}},$
$$C\in \mathcal{C}(u,i,j)\iff \gamma_{u,v}C \in \mathcal{C}(v,i,j).$$
Similarly

$$C\in \mathcal{C}(u,1,2)\iff h_{u}C \in \mathcal{C}(u,2,3)\iff h_{u}^{2}C\in\mathcal{C}(u,1,3).$$
Let $n(C)=\vert\{C'\in\tilde{\mathcal{C}}\;:\;C'=C\}\vert,$ i.e. $n(C)$ is the multiplicity of $C$ in $\tilde{\mathcal{C}}$. By applying $\gamma_{v_{1},v}$ and $h_{v_{1}}$, we see that for any $v\in V_{1}$ and $i\neq j,i'\neq j'$:   
    $$\sum\limits_{C\in\mathcal{C}(v_{1},i,j)}n(C)=\sum\limits_{C\in\mathcal{C}(v,i',j')}n(C).$$
Therefore there exists an integer $M_{1}$ such that for for any $v\in V_{1}$ and $i\neq j$:
    $$\sum\limits_{C\in\mathcal{C}(v,i,j)}n(C)=\frac{M_{1}}{3}.$$
Similarly there exists an integer $M_{2}$ such that for any $v\in V_{2}$ and $i\neq j$:   
    $$\sum\limits_{C\in\mathcal{C}(v,i,j)}n(C)=\frac{M_{2}}{3}.$$
    
    Now finally we wish to show that $M_{1}=M_{2}$. However, this follows immediately by construction, as $\mathcal{B}=\gamma \cdot\mathcal{A}$, and $\mathcal{C}=\mathcal{A}\cup \mathcal{B}$.
\end{proof}
Using this, we investigate the separation of several graphs.

    \begin{table}[H]
    \centering
    \small
    \begin{tabular}{c|ccl}
  $x_{i}$  &\multicolumn{3}{c}{$x_{j}$ adjacent to $x_{i}$}\\
    \hline
0&1& 2& 3\\
1&0& 4& 5\\
2&0& 6& 8\\
3&0& 7& 9\\
4&1& 11& 14\\
5&1& 10& 13\\
6&2& 12& 16\\
7&3& 12& 15
\end{tabular}
    \begin{tabular}{c|ccl}
  $x_{i}$  &\multicolumn{3}{c}{$x_{j}$ adjacent to $x_{i}$}\\
    \hline
8&2& 11& 18\\
9&3& 10& 17\\
10&5& 9& 21\\
11&4& 8& 20\\
12&6& 7& 19\\
13&5& 19& 23\\
14&4& 19& 22\\
15&7& 20& 23
\end{tabular}
    \begin{tabular}{c|ccl}
  $x_{i}$  &\multicolumn{3}{c}{$x_{j}$ adjacent to $x_{i}$}\\
    \hline
16&6& 21& 22\\
17&9& 20& 22\\
18&8& 21& 23\\
19&12& 13& 14\\
20&11& 15& 17\\
21&10& 16& 18\\
22&14& 16& 17\\
23&13& 15& 18
\end{tabular}
    \caption{Edge incidences for $F24A$}
    \label{tab:F24 edges}
\end{table}
\begin{lemma}
The graph $F24A$ is $\dagger$-separated.

\end{lemma}
\begin{proof}
By a computer search we find all $3$-separated vertex cutsets in $F24A$:
  \begin{table}[H]
 \centering
    \begin{tabular}{l}
$C_{1}=\{x_{0}, x_{10}, x_{11}, x_{12},x_{22},x_{23}\}$,\\
$C_{2}=\{x_{1}, x_{8}, x_{9}, x_{15}, x_{16}, x_{19}\}$,\\
$C_{3}=\{x_{2}, x_{4}, x_{7}, x_{13},x_{17},x_{21}\}$,\\
$C_{4}=\{x_{3}, x_{5}, x_{6}, x_{14}, x_{18}, x_{19}\}$.
  \end{tabular}
  \end{table}
We note $diam(F24A)=4$. As the above are disjoint and proper, it follows easily that $F24A$  is $\dagger$-separated.
\end{proof}
\vspace*{-10pt}
\begin{table}[H]
\centering
\small
    \begin{tabular}{c|ccl}
  $x_{i}$  &\multicolumn{3}{c}{$x_{j}$ adjacent to $x_{i}$}\\
    \hline
0&1& 2& 3\\
1&0& 4& 7\\
2&0& 6& 9\\
3&0& 5& 8\\
4&1& 10& 13\\
5&3& 11& 14\\
6&2& 12& 15\\
7&1& 11& 16\\
8&3& 12& 17
\end{tabular}
    \begin{tabular}{c|ccl}
  $x_{i}$  &\multicolumn{3}{c}{$x_{j}$ adjacent to $x_{i}$}\\
    \hline
9&2& 10& 18\\
10&4& 9& 22\\
11&5& 7& 20\\
12&6& 8& 21\\
13&4& 23& 24\\
14&5& 24& 25\\
15&6& 23& 25\\
16&7& 21& 23\\
17&8& 22& 24
\end{tabular}
    \begin{tabular}{c|ccl}
  $x_{i}$  &\multicolumn{3}{c}{$x_{j}$ adjacent to $x_{i}$}\\
    \hline
18&9& 20& 25\\
19&20& 21& 22\\
20&11& 18& 19\\
21&12& 16& 19\\
22&10& 17& 19\\
23&13& 15& 16\\
24&13& 14& 17\\
25&14& 15& 18
\end{tabular}
    \caption{Edge incidences for $F26A$}
    \label{tab:F26 edges}
\end{table}
\begin{lemma}\label{lem: F26A is * separated}
The graph $F26A$ is $\Large{*}$-separated.
\end{lemma}
\begin{proof}
We can take $v_{1}=x_{0}$, $w_{i}=x_{i}$. $D(x_{3})=\emptyset$, as $diam(F26A)=4$. Using the notation as in Lemma \ref{lem: simple condition for 3 sep} we find $A_{1}=\{x_{0}, x_{10}, x_{12}, x_{14}, x_{20}, x_{23}\}.$ The result follows by Lemma \ref{lem: simple condition for 3 sep}.
\end{proof}

We defer the incidence table of $F40A$, and the collection of cutsets found, to Appendix \ref{section: cutset appendix}.
\begin{lemma}\label{lem: F40 is separated}
The graph $F40A$ is weighted strongly edge $3$-separated.
\end{lemma}
\begin{proof}
We require a large number of cutsets for this proof: they can be found in Appendix \ref{section: cutset appendix}. 

In particular, we find a collection of cutsets $\{C_{i}\}_{i}$ such that for any vertices $w_{1},w_{2}$ with $d(x_{0},w_{1})\geq 3$ and $d(w_{1},w_{2})=1$ there exists some $C_{i}$ separating $\{x_{0},x_{1}\}$ and $\{w_{1},w_{2}\}$ (this can be easily checked by computer).
Similarly for any vertices $w_{1},w_{2}$ with $d(x_{1},w_{1})\geq 3$ and $d(w_{1},w_{2})=1$ there exists some $C_{i}$ separating $\{x_{0},x_{1}\}$ and $\{w_{1},w_{2}\}$. By passing to subsets of $C_{i}$ we may assume each of these cutsets are minimal and therefore proper.

As $Aut(F40)$ acts edge and vertex transitively, it follows by Lemma \ref{lem: strong edge sep condition} that $F40$ is strongly edge $3$-separated.

By Lemma \ref{lem: edge transitive automorphism group can solve equations}, $F40A$ is weighted disjointly strongly edge $3$-separated.
\end{proof}

 \begin{table}[H]
\centering
 \small
    \begin{tabular}{c|ccl}
  $x_{i}$  &\multicolumn{3}{c}{$x_{j}$ adjacent to $x_{i}$}\\
    \hline
0&1& 2& 3\\
1&0& 4& 5\\
2&0& 6& 8\\
3&0& 7& 9\\
4&1& 11& 17\\
5&1& 10& 16\\
6&2& 13& 21\\
7&3& 12& 20\\
8&2& 15& 19\\
9&3& 14& 18\\
10&5& 23& 25\\
11&4& 22& 24\\
12&7& 23& 29\\
13&6& 22& 28\\
14&9& 22& 27\\
15&8& 23& 26
\end{tabular}
    \begin{tabular}{c|ccl}
  $x_{i}$  &\multicolumn{3}{c}{$x_{j}$ adjacent to $x_{i}$}\\
    \hline
16&5& 27& 31\\
17&4& 26& 30\\
18&9& 25& 35\\
19&8& 24& 34\\
20&7& 28& 33\\
21&6& 29& 32\\
22&11& 13& 14\\
23&10& 12& 15\\
24&11& 19& 43\\
25&10& 18& 42\\
26&15& 17& 47\\
27&14& 16& 46\\
28&13& 20& 45\\
29&12& 21& 44\\
30&17& 40& 46\\
31&16& 39& 47
\end{tabular}
    \begin{tabular}{c|ccl}
  $x_{i}$  &\multicolumn{3}{c}{$x_{j}$ adjacent to $x_{i}$}\\
    \hline
32&21& 41& 43\\
33&20& 41& 42\\
34&19& 40& 44\\
35&18& 39& 45\\
36&43& 45& 46\\
37&42& 44& 47\\
38&39& 40& 41\\
39&31& 35& 38\\
40&30& 34& 38\\
41&32& 33& 38\\
42&25& 33& 37\\
43&24& 32& 36\\
44&29& 34& 37\\
45&28& 35& 36\\
46&27& 30& 36\\
47&26& 31& 37
\end{tabular}
    \caption{Edge incidences for $F48A$}
    \label{tab:F48 edges}
\end{table}

\begin{lemma}
The graph $F48A$ is $\dagger$-separated.
\end{lemma}
\begin{proof}
By a computer search we find all $3$-separated vertex cutsets in $F48A$:
  \begin{table}[H]
 \centering
    \begin{tabular}{l}
$C_{1}=\{x_{0}, x_{16}, x_{17}, x_{18}, x_{19}, x_{20}, x_{21}, x_{22}, x_{23}, x_{36}, x_{37}, x_{38}\}$,\\
$C_{2}=\{x_{1}, x_{6}, x_{7}, x_{14}, x_{15}, x_{24}, x_{25}, x_{30}, x_{31}, x_{41}, x_{44}, x_{45}\}$,\\
$C_{3}=\{x_{2}, x_{5}, x_{9}, x_{11}, x_{12}, x_{26}, x_{28}, x_{32}, x_{34}, x_{39}, x_{42}, x_{46}\}$,\\
$C_{4}=\{x_{3}, x_{4}, x_{8}, x_{10}, x_{13}, x_{27}, x_{29}, x_{33}, x_{35}, x_{40}, x_{43}, x_{47}\}$.
  \end{tabular}
  \end{table}
The above are disjoint and proper, and it can be seen that $F48A$  is $\dagger$-separated.
\end{proof}
We defer the incidence table of $G54$, and the collection of cutsets found, to Appendix \ref{section: cutset appendix}.
\begin{lemma}\label{lem: G54 is separated}
    The Gray Graph $G54$ is strongly edge $3$-separated.
\end{lemma}
\begin{proof}
We require a large number of cutsets for this proof: they can be found in Appendix \ref{section: cutset appendix}. In particular, we find a collection of $3$-separated cutsets $\{C_{i}\}_{i}$ such that each $C_{i}$ contains one of the edges $$(x_{0},x_{1}),(x_{0},x_{53}),(x_{24},x_{25}),(x_{25},x_{26}).$$
Therefore, as each cutset is $3$-separated, they cannot contain the edge $(x_{0},x_{25})$, and so for each cutset, $x_{0}$ and $x_{25}$ lie in the same component of $G54- C_{i}$.

We also show that for any point $v$ with $d(x_{0},v)\geq 3$ and any neighbour $w$ of $v$, there exists some $C_{i}$ separating $\{x_{0},x_{25}\}$ and $\{v,w\}$. Furthermore, for any point $v$ with $d(x_{25},v)\geq 3$ and any neighbour $w$ of $v$, there exists some $C_{i}$ separating $\{x_{0},x_{25}\}$ and $\{v,w\}$.
By passing to subsets of $C_{i}$ we may assume each of these cutsets are minimal and therefore proper.
Now, let $p=(u,w_{1})(w_{1},w_{2})\hdots(w_{n},v)$ be some path with $2\leq n\leq 4$ of length between $3$ and $6$. Let $u'$ be adjacent to $u$ and $v'$ be adjacent to $v$.

Note again that $Aut(G54)$ acts transitively on the set of edges.
If we can map $u$ to $x_{0}$ by some element $\gamma\in Aut(G54)$, then we may also map $u'$ to $x_{25}$ by $\gamma$, and then for some $i$, $C_{i}$ separates $x_{0},x_{25}$ and $\gamma v,\gamma v'$: $\gamma^{-1}C_{i}$ then separates $u,u'$ and $v,v'$. Otherwise, we map $u$ to $x_{25}$ by $\gamma$, so that $\gamma u'=x_{0}$. The result follows similarly.

Therefore, $G54$ is strongly edge $3$-separated, and as it has an edge transitive automorphism group, it is weighted strongly edge $3$-separated by Lemma \ref{lem: edge transitive automorphism group can solve equations}.
\end{proof}
We finally need to prove the following.
\begin{lemma}\label{lem: splicing together generalized triangle groups}
Let $Y$ be a finite triangle complex such that each triangle is a unit equilateral Euclidean triangle. Suppose that the link of each vertex is either $\Large{*}$-separated or $\dagger$-separated with the combinatorial metric (we allow a mixture of these). Then $Y$ is gluably $\pi$-separated.
\end{lemma}
\begin{proof}
It is clear that $Y$ is nonpositively curved and regular. By Lemma \ref{lem: solving gluing equations for minimal edge cutsets}, if the link of each vertex is $\dagger$-separated with the combinatorial metric then we are finished.

Otherwise, let $\{v_{k}\}$ be the vertices such that $Lk(v_{k})$ is $\Large{*}$-separated with the combinatorial metric, and $\{w_{l}\}$ be the vertices such that $Lk(w_{l})$ is $\dagger$-separated with the combinatorial metric.
Note that a $3$-separated cutset in $Lk(x)$ under the combinatorial metric is a $\pi$-separated cutset in $Lk(x)$ under the combinatorial metric.

For each proper $\pi$-separated cutset $C$ in $Lk(w_{l})$ we may assign the three partitions $P_{1}(C),P_{2}(C),P_{3}(C)$ corresponding to placing two components of $Lk(w_{l})-C$ in the same element of the partition. For each cutset $C$ in $Lk(v_{k})$ assign the unique partition of connectedness of $Lk(v_{k})-C$.

Since the links are $\dagger$-separated and $\Large{*}$-separated, by assumption for each vertex $x\in Y$ there exists a positive integer $N_{x}>0$ and a system of strictly positive weights $n_{x}(C)$ for $C\in \mathcal{C}_{x}$ such that for any vertex $e$ in $Lk_{Y}(x)$,

$$\sum\limits_{C\in \mathcal{C}(e)}n_{x}(C)=\sum\limits_{C\in \mathcal{C}(e)\cap \mathcal{C}_{x}}n_{x}(C)=N_{x}.$$
Furthermore, if $Lk(v_{k})$ is $\Large{*}$-separated, then for any vertex $y\in V(Lk(v_{k}))$ and $i\neq j$
  $$\sum\limits_{C\in\mathcal{C}_{v_{k}}(y,i,j)}n_{v_{l}}(C)=\frac{N_{v_{l}}}{3}.$$
Let $M=\prod_{x\in V(Y)}N_{x},$ and for a cutset $C\in\mathcal{C}_{x},$ define $$m(C)=Mn_{x}(C)\slash N_{x}.$$ It follows that for an edge $e$ in $Lk_{G\backslash X}(x)$, 
$$\sum\limits_{C\in \mathcal{C}(e)}m(C)=\frac{M}{N_{x}}\sum\limits_{C\in \mathcal{C}(e)}n_{x}(C)=\frac{M}{N_{x}}N_{x}=M.$$
Now, take $\mu(C,P(C))=m(C)$. It follows that for any oriented edge $e$ of $Y^{(1)}$ starting at some $w_{l}$ and any partition $(C,P)\in \mathcal{CP}(e)$:
 
$$\sum \limits_{(C',P')\in [C,P]_{e}}\mu (C',P')=\sum \limits_{(C',P')\in [C,P]_{e}}m(C')=\frac{1}{3}\sum \limits_{C'\in \mathcal{C}(e)}m(C')=\frac{1}{3}M.$$
Similarly, by the definition of $\Large{*}$-separated, for each $v_{k}$, each edge $e$ starting at $v_{k}$,  and $(C,P(C))\in \mathcal{C}(e),$

$$\sum \limits_{(C',P')\in [C,P]_{e}}\mu (C',P')=\sum \limits_{(C',P')\in [C,P]_{e}}m(C')=\frac{1}{3}\sum \limits_{(C',P')\in \mathcal{C}(e)}m(C')=\frac{1}{3}M,$$
and so the gluing equations are solved.
\end{proof}

The results of Corollary \ref{coralph: small girth generalized triangle groups} now follow from \cite[Theorem 3.1]{Caprace-Conder-Kaluba-Witzel_triangle}, \cite[Proposition 3.2]{Lubotzky-Manning-Wilton}, the above lemmas concerning the separation of the graphs considered, Theorem \ref{mainthm: cubulating groups}, and Theorem \ref{mainthm: cubulating generalized triangle groups}.
%----------------------------------------------------------------------------------------
%   SUBSECTION: CUBULATING DEHN FILLINGS OF GENERALIZED ORDINARY TRIANGLE GROUPS
%----------------------------------------------------------------------------------------
	\subsection{Cubulating Dehn fillings of generalized ordinary triangle groups}\label{subsection: cubulating dehn fillings of generalized triangle groups}
	We now apply Theorem \ref{mainthm: cubulating groups} to the generalized triangle groups of \cite{Lubotzky-Manning-Wilton}, in particular retrieving consequences of the malnormal special quotient theorem of Wise \cite{Wise-MSQT}. 
	
	\begin{cor}\label{mainthm: cubulating dehn fillings of generalized triangle groups}
		Let $\Gamma_{i}\looparrowright C_{k,2}$ be finite $n(i)$-sheeted normal covering graphs. There exist finite-sheeted normal covering graphs $\dot{\Gamma}_{i}\looparrowright \Gamma_{i}$ of index at most
		$$4\bigg(4^{4^{kn(i)}}\bigg)$$ 
		such that for any collection of finite-sheeted covering graphs $\Delta_{i}\looparrowright\Gamma_{i}$ that factor as $\Delta_{i}\looparrowright\dot{\Gamma}_{i}\looparrowright\Gamma_{i},$ and any $j$, the group $G^{j}_{0,k}(\Delta_{1},\Delta_{2},\Delta_{3})$ is hyperbolic and acts properly discontinuously and cocompactly on a $CAT(0)$ cube complex.
	\end{cor}
	
	We consider covers of $\sigma$-separated graphs: we restrict our consideration to graphs with the combinatorial metric. We note the following lemma.
	\begin{lemma}\label{lem:cover of separation}
		Let $p:\tilde{\Gamma}\looparrowright\Gamma$ be a covering graph. Let $e\in E(\Gamma)$ and let $\tilde{e}_{1},\tilde{e}_{2}\in p^{-1}(e)$ be distinct. Then $$d_{\tilde{\Gamma}}(m(\tilde{e}_{1}),m(\tilde{e}_{2}))\geq girth(\Gamma).$$
	\end{lemma}
	We now show that covers of $\sigma$-separated graphs are also $\sigma$-separated.
	\begin{lemma}\label{lem: covers of suitable graphs}
		Let $\Gamma$ be a weighted (disjointly) edge $\sigma$-separated graph with $girth(\Gamma)\geq \sigma $ and $p:\tilde{\Gamma}\looparrowright\Gamma$ a finite-sheeted covering graph. Then $\tilde{\Gamma}$ is also weighted (disjointly) edge $\sigma$-separated, and $girth(\tilde{\Gamma})\geq girth(\Gamma)$.
	\end{lemma}
	\begin{proof}
		It is clear that $girth(\tilde{\Gamma})\geq girth(\Gamma)$, and that $\tilde{\Gamma}$ is connected and contains no vertices of degree $1$. Let $\mathcal{C}_{1}, \hdots , \mathcal{C}_{m}\subseteq E(\Gamma)$ be the $\sigma$-separated cut sets of $\Gamma$. Let $\tilde{\mathcal{C}}_{i}=p^{-1}(\mathcal{C}_{i})$: by Lemma \ref{lem:cover of separation}, and by noting that for all $x,y\in\tilde{\Gamma}$ we have $d_{\tilde{\Gamma}}(x,y)\geq d_{\Gamma}(p(x),p(y))$, we see that $\tilde{\mathcal{C}}_{i}$ is a collection of proper $\min \{girth (\Gamma ),\sigma \}$-separated cut sets. Furthermore $\vert \tilde{C}_{i}\vert\geq \vert C_{i}\vert\geq 2$ As $girth(\Gamma)\geq \sigma$, these are $\sigma$-separated and $\cup_{i}\tilde{\mathcal{C}}_{i}=E(\tilde{\Gamma})$. Therefore, $\tilde{\Gamma}$ is edge $\sigma$-separated.
		
		If $\Gamma$ is disjointly separated, it is clear that $\tilde{\Gamma}$ is disjointly separated. Finally, defining $n(\tilde{C}_{i})=n(C_{i})$, it can be seen that the weight equations are satisfied, so that $\tilde{\Gamma}$ is weighted (disjointly) edge $\sigma$-separated.
	\end{proof}
	
	Using the above, we wish to show that given any graph $\Gamma$, there exists a finite-sheeted $3$-separated covering graph $\tilde{\Gamma}\looparrowright \Gamma$.
\begin{definition}
		Let $\Gamma$ be a graph and $m\geq 0$. The \emph{$\mathbb{Z}_{m}$ cover of $\Gamma$}, $$p_{m}:\mathbb{Z}_{m} (\Gamma )\looparrowright\Gamma,$$ is the $m^{b_{1}(\Gamma)}$-sheeted cover corresponding to the kernel of the canonical map 
		$\pi_{1}(\Gamma)\rightarrow H_{1}(\Gamma,\mathbb{Z}_{m}).$
	\end{definition}
	The use of this is the following.
	\begin{lemma}\label{lem: Gamma2 is suitable}
		Let $\Gamma$ be a finite connected graph with no cut edges and let $m\geq 1$. The covering graph $\mathbb{Z}_{2m} ( \Gamma )$ is weighted disjointly edge $girth(\Gamma)$-separated and $girth ( \mathbb{Z}_{2m} ( \Gamma ))= 2m (girth(\Gamma)).$
	\end{lemma}
	\begin{proof}
		Let $e\in E(\Gamma)$. We claim $p_{2m}^{-1}(e)$ is a proper $girth(\Gamma)$-separated cut set in $\mathbb{Z}_{2m} ( \Gamma )$. By Lemma \ref{lem:cover of separation}, $p_{2m}^{-1}(e)$ is $girth(\Gamma)$-separated. It suffices to show that if two points $x$ and $y$ are joined by a path $q$ containing one edge of $p_{2m}^{-1}(e)$, then any path $q'$ between them contains an edge of $p_{2m}^{-1}(e)$. Now suppose not: consider such a path $q'$ not containing any edge of $p_{2m}^{-1}(e)$, and consider the loop $qq'$. Then $p_{2m}(qq')$ is trivial in the map to $H_{1}(\Gamma, \mathbb{Z}_{2m})$, so is homotopic to a curve containing $e$ an even number of times, a contradiction.

		Therefore the set $\mathcal{C}_{e}=\{p_{2m}^{-1}(e)\;:\; e\in E(\Gamma)\}$ is a disjoint collection of proper $girth(\Gamma)$-separated edge-cut sets such that any edge in $\mathbb{Z}_{2m}(\Gamma)$ appears exactly one cut set: the weight equations are trivially satisfied and so $\mathbb{Z}_{2m}(\Gamma)$ is weighted disjointly $girth(\Gamma)$-separated.

		Any loop in $\mathbb{Z}_{2m} ( \Gamma )$ projects to a loop homotopic to a product of loops where each loop is traversed $2m$ times, and so $girth(\mathbb{Z}_{2} ( \Gamma )) = 2m( girth(\Gamma))$.
	\end{proof}

	Using this, we prove the following.
	\begin{proof}[Proof of Corollary \ref{mainthm: cubulating dehn fillings of generalized triangle groups}]
	Let $\Gamma_{i}\looparrowright C_{k,2}$ be $n(i)$-sheeted normal covering graphs . Let $\dot{\Gamma}_{i}:=\mathbb{Z}_{2} (\mathbb{Z}_{2} ( \Gamma ))$: these are $$2^{2-(2^{2-2n(i)+kn(i)}+2)n(i)+(2^{1-2n(i)+kn(i)}+1)kn(i)}\leq 4^{1+kn(i)2^{kn(i)}}\leq 4\bigg(4^{4^{kn(i)}}\bigg)-$$ sheeted covering graphs, which, by Lemma \ref{lem: Gamma2 is suitable}, are weighted disjointly edge $3$-separated under the combinatorial metric and have girth at least $8$. Furthermore, it is clear that $\dot{\Gamma_{i}}\looparrowright C_{k,2}$ are normal covers. Suppose $\Delta_{i}\looparrowright\Gamma_{i}$ factors as $\Delta_{i}\looparrowright\dot{\Gamma}_{i}\looparrowright\Gamma_{i}$. 
		By \cite[Proposition 3.2]{Lubotzky-Manning-Wilton}, noting that $girth(\Delta_{i})\geq girth(\dot{\Gamma_{i}})> 6$, the group $G^{j}_{0,k}(\Delta_{1},\Delta_{2},\Delta_{3})$ is hyperbolic. The $\Delta_{i}$ are covers of $\dot{\Gamma}_{i}$, so by Lemma \ref{lem: covers of suitable graphs} are also weighted edge $3$-separated under the combinatorial metric. The result now follows from Lemma \ref{lem: solving gluing equations for minimal edge cutsets}, \cite[Proposition 3.2]{Lubotzky-Manning-Wilton}, and Theorem \ref{mainthm: cubulating groups}.
	\end{proof}
\appendix
%----------------------------------------------------------------------------------------
%   Large collections of cutsets
%---------------------------------------------------------------------------------------
\section{Large collections of cutsets}\label{section: cutset appendix}
In this section we provide the selection of cutsets described in Lemmas \ref{lem: F40 is separated} and \ref{lem: G54 is separated}. Throughout, we use the notation $e_{i,j}=(x_{i},x_{j})$.
%----------------------------------------------------------------------------------------
%   F40A
%---------------------------------------------------------------------------------------
\subsection{F40A}

The graph $F40A$ has the following edge incidences. 
 \begin{table}[H]
    \begin{tabular}{c|ccl}
  $x_{i}$  &\multicolumn{3}{c}{$x_{j}$ adjacent to $x_{i}$}\\
    \hline
0&1& 2& 3\\
1&0& 4& 5\\
2&0& 6& 7\\
3&0& 8& 9\\
4&1& 10& 12\\
5&1& 11& 13\\
6&2& 15& 18\\
7&2& 14& 19\\
8&3& 17& 20\\
9&3& 16& 21\\
10&4& 23& 31\\
11&5& 22& 30\\
12&4& 25& 29\\
13&5& 24& 28
\end{tabular}
\begin{tabular}{c|ccl}
  $x_{i}$  &\multicolumn{3}{c}{$x_{j}$ adjacent to $x_{i}$}\\
    \hline
14&7& 23& 33\\
15&6& 22& 32\\
16&9& 25& 33\\
17&8& 24& 32\\
18&6& 27& 29\\
19&7& 26& 28\\
20&8& 26& 31\\
21&9& 27& 30\\
22&11& 15& 35\\
23&10& 14& 34\\
24&13& 17& 34\\
25&12& 16& 35\\
26&19& 20& 35
\end{tabular}
\begin{tabular}{c|ccl}
  $x_{i}$  &\multicolumn{3}{c}{$x_{j}$ adjacent to $x_{i}$}\\
    \hline
27&18& 21& 34\\
28&13& 19& 36\\
29&12& 18& 36\\
30&11& 21& 37\\
31&10& 20& 37\\
32&15& 17& 38\\
33&14& 16& 38\\
34&23& 24& 27\\
35&22& 25& 26\\
36&28& 29& 39\\
37&30& 31& 39\\
38&32& 33& 39\\
39&36& 37& 38
\end{tabular}
    \caption{Edge incidences for $F40A$}
    \label{tab:F40 edges}
\end{table}
We find the following cutsets.\\
$\noindent\left\{e_{1,5},e_{2,7},e_{3,9},e_{10,23},e_{12,25},e_{15,22},e_{17,24},e_{18,27},e_{20,26},e_{28,36},e_{30,37},e_{33,38}\right\},\\
\left\{e_{1,4},e_{2,6},e_{3,8},e_{11,22},e_{13,24},e_{14,23},e_{16,25},e_{19,26},e_{21,27},e_{29,36},e_{31,37},e_{32,38}\right\},\\
\left\{e_{0,3},e_{4,12},e_{5,13},e_{6,18},e_{7,19},e_{16,33},e_{17,32},e_{20,31},e_{21,30},e_{22,35},e_{23,34},e_{36,39}\right\},\\
\left\{e_{0,2},e_{4,10},e_{5,11},e_{8,20},e_{9,21},e_{14,33},e_{15,32},e_{18,29},e_{19,28},e_{24,34},e_{25,35},e_{37,39}\right\},\\
\left\{e_{0,1},e_{6,15},e_{7,14},e_{8,17},e_{9,16},e_{10,31},e_{11,30},e_{12,29},e_{13,28},e_{26,35},e_{27,34},e_{38,39}\right\},\\
\left\{e_{2,7},e_{3,8},e_{4,10},e_{5,13},e_{15,32},e_{16,33},e_{26,35},e_{27,34},e_{29,36},e_{30,37}\right\},\\
\left\{e_{2,6},e_{3,9},e_{4,12},e_{5,11},e_{14,33},e_{17,32},e_{26,35},e_{27,34},e_{28,36},e_{31,37}\right\},\\
\left\{e_{1,5},e_{3,8},e_{6,15},e_{7,19},e_{10,31},e_{21,30},e_{24,34},e_{25,35},e_{29,36},e_{33,38}\right\},\\
\left\{e_{1,5},e_{2,6},e_{8,17},e_{9,21},e_{12,29},e_{19,28},e_{22,35},e_{23,34},e_{31,37},e_{33,38}\right\},\\
\left\{e_{1,4},e_{3,9},e_{6,18},e_{7,14},e_{11,30},e_{20,31},e_{24,34},e_{25,35},e_{28,36},e_{32,38}\right\},\\
\left\{e_{1,4},e_{2,7},e_{8,20},e_{9,16},e_{13,28},e_{18,29},e_{22,35},e_{23,34},e_{30,37},e_{32,38}\right\},\\
\left\{e_{0,3},e_{5,11},e_{6,15},e_{10,31},e_{12,25},e_{14,33},e_{17,24},e_{19,26},e_{21,27},e_{36,39}\right\},\\
\left\{e_{0,3},e_{4,10},e_{7,14},e_{11,30},e_{13,24},e_{15,32},e_{16,25},e_{18,27},e_{20,26},e_{36,39}\right\},\\
\left\{e_{0,2},e_{5,13},e_{8,17},e_{10,23},e_{12,29},e_{15,22},e_{16,33},e_{19,26},e_{21,27},e_{37,39}\right\},\\
\left\{e_{0,2},e_{4,12},e_{9,16},e_{11,22},e_{13,28},e_{14,23},e_{17,32},e_{18,27},e_{20,26},e_{37,39}\right\},\\
\left\{e_{0,1},e_{7,19},e_{8,20},e_{10,23},e_{13,24},e_{15,22},e_{16,25},e_{18,29},e_{21,30},e_{38,39}\right\},\\
\left\{e_{0,1},e_{6,18},e_{9,21},e_{11,22},e_{12,25},e_{14,23},e_{17,24},e_{19,28},e_{20,31},e_{38,39}\right\}.$

%----------------------------------------------------------------------------------------
%   G54
%---------------------------------------------------------------------------------------
\subsection{G54}
The graph $G54$ has the follwing edge incidences. 
 \begin{table}[H]
    \begin{tabular}{c|ccl}
  $x_{i}$  &\multicolumn{3}{c}{$x_{j}$ adjacent to $x_{i}$}\\
    \hline
0&1&25&53\\
1&0&2&30\\ 
2&1&3&15\\
3&2&4&44\\ 
4&3&5&11\\ 
5&4&6&52\\ 
6&5&7&31\\ 
7&6&8&36\\ 
8&7&9&21\\
9&8&10&50\\
10&9&11&17\\
11&4&10&12\\
12&11&13&37\\
13&12&14&42\\
14&13&15&27\\
15&2&14&16\\
16&15&17&23\\
17&10&16&18
\end{tabular}
    \begin{tabular}{c|ccl}
  $x_{i}$  &\multicolumn{3}{c}{$x_{j}$ adjacent to $x_{i}$}\\
    \hline
18&17&19&43\\
19&18&20&48\\
20&19&21&33\\
21&8&20&22\\
22&21&23&29\\
23&16&22&24\\
24&23&25&49\\
25&0&24&26\\
26&25&27&39\\
27&14&26&28\\
28&27&29&35\\
29&22&28&30\\
30&1&29&31\\
31&6&30&32\\
32&31&33&45\\
33&20&32&34\\
34&33&35&41\\
35&28&34&36
\end{tabular}
    \begin{tabular}{c|ccl}
  $x_{i}$  &\multicolumn{3}{c}{$x_{j}$ adjacent to $x_{i}$}\\
    \hline
36&7&35&37\\
37&12&36&38\\
38&37&39&51\\
39&26&38&40\\
40&39&41&47\\
41&34&40&42\\
42&13&41&43\\
43&18&42&44\\
44&3&43&45\\
45&32&44&46\\
46&45&47&53\\
47&40&46&48\\
48&19&47&49\\
49&24&48&50\\
50&9&49&51\\
51&38&50&52\\
52&5&51&53\\
53&0&46&52
\end{tabular}   \caption{Edge incidences for $G54$}\end{table}

We find the following cutsets.\\
$\left\{e_{0,53},e_{2,3},e_{13,14},e_{16,17},e_{21,22},e_{24,49},e_{26,39},e_{28,35},e_{30,31}\right\},\\\left\{e_{0,1},e_{3,4},e_{6,31},e_{8,21},e_{10,17},e_{12,13},e_{19,48},e_{23,24},e_{26,27},e_{35,36},e_{38,51},e_{40,41},e_{45,46}\right\},\\\left\{e_{0,1},e_{3,4},e_{6,31},e_{8,21},e_{10,17},e_{12,13},e_{19,48},e_{23,24},e_{26,27},e_{35,36},e_{40,41},e_{45,46},e_{50,51}\right\},\\\left\{e_{0,1},e_{3,4},e_{6,31},e_{8,21},e_{10,17},e_{12,13},e_{19,48},e_{23,24},e_{26,27},e_{35,36},e_{40,41},e_{45,46},e_{51,52}\right\},\\\left\{e_{0,1},e_{3,4},e_{6,31},e_{8,21},e_{10,17},e_{14,15},e_{19,48},e_{23,24},e_{28,29},e_{33,34},e_{37,38},e_{42,43},e_{45,46}\right\},\\\left\{e_{0,1},e_{3,4},e_{6,31},e_{8,21},e_{10,17},e_{14,15},e_{19,48},e_{23,24},e_{28,29},e_{33,34},e_{38,39},e_{42,43},e_{45,46}\right\},\\\left\{e_{0,1},e_{3,4},e_{6,31},e_{8,21},e_{10,17},e_{14,15},e_{19,48},e_{23,24},e_{28,29},e_{33,34},e_{38,51},e_{42,43},e_{45,46}\right\},\\\left\{e_{0,1},e_{3,4},e_{6,31},e_{9,50},e_{12,13},e_{15,16},e_{18,43},e_{20,33},e_{22,29},e_{26,27},e_{35,36},e_{40,41},e_{45,46}\right\},\\\left\{e_{0,1},e_{3,4},e_{6,31},e_{12,13},e_{15,16},e_{18,43},e_{20,33},e_{22,29},e_{26,27},e_{35,36},e_{40,41},e_{45,46},e_{49,50}\right\},\\\left\{e_{0,1},e_{3,4},e_{6,31},e_{12,13},e_{15,16},e_{18,43},e_{20,33},e_{22,29},e_{26,27},e_{35,36},e_{40,41},e_{45,46},e_{50,51}\right\},\\\left\{e_{0,1},e_{3,44},e_{5,52},e_{7,8},e_{10,11},e_{13,42},e_{15,16},e_{19,48},e_{22,29},e_{26,27},e_{31,32},e_{34,35},e_{37,38}\right\},\\\left\{e_{0,1},e_{3,44},e_{5,52},e_{7,8},e_{10,11},e_{13,42},e_{15,16},e_{22,29},e_{26,27},e_{31,32},e_{34,35},e_{37,38},e_{47,48}\right\},\\\left\{e_{0,1},e_{3,44},e_{5,52},e_{7,8},e_{10,11},e_{13,42},e_{15,16},e_{22,29},e_{26,27},e_{31,32},e_{34,35},e_{37,38},e_{48,49}\right\},\\\left\{e_{0,1},e_{3,44},e_{5,52},e_{7,36},e_{9,50},e_{11,12},e_{14,15},e_{17,18},e_{20,21},e_{23,24},e_{28,29},e_{31,32},e_{39,40}\right\},\\\left\{e_{0,1},e_{3,44},e_{5,52},e_{7,36},e_{9,50},e_{11,12},e_{14,15},e_{17,18},e_{20,21},e_{23,24},e_{28,29},e_{31,32},e_{40,41}\right\},\\\left\{e_{0,1},e_{3,44},e_{5,52},e_{7,36},e_{9,50},e_{11,12},e_{14,15},e_{17,18},e_{20,21},e_{23,24},e_{28,29},e_{31,32},e_{40,47}\right\},\\\left\{e_{0,1},e_{3,44},e_{5,52},e_{9,50},e_{13,42},e_{17,18},e_{20,21},e_{23,24},e_{26,27},e_{31,32},e_{34,35},e_{37,38},e_{40,47}\right\},\\\left\{e_{0,1},e_{3,44},e_{5,52},e_{9,50},e_{13,42},e_{17,18},e_{20,21},e_{23,24},e_{26,27},e_{31,32},e_{34,35},e_{37,38},e_{46,47}\right\},\\\left\{e_{0,1},e_{3,44},e_{5,52},e_{9,50},e_{13,42},e_{17,18},e_{20,21},e_{23,24},e_{26,27},e_{31,32},e_{34,35},e_{37,38},e_{47,48}\right\},\\\left\{e_{0,53},e_{2,3},e_{5,6},e_{8,9},e_{11,12},e_{16,17},e_{19,20},e_{24,49},e_{27,28},e_{32,45},e_{38,51},e_{40,47},e_{42,43}\right\},\\\left\{e_{0,53},e_{2,3},e_{5,6},e_{8,9},e_{11,12},e_{16,17},e_{19,20},e_{24,49},e_{28,29},e_{32,45},e_{38,51},e_{40,47},e_{42,43}\right\},\\\left\{e_{0,53},e_{2,3},e_{5,6},e_{8,9},e_{11,12},e_{16,17},e_{19,20},e_{24,49},e_{28,35},e_{32,45},e_{38,51},e_{40,47},e_{42,43}\right\},\\\left\{e_{0,53},e_{2,3},e_{5,6},e_{8,9},e_{13,14},e_{16,17},e_{19,20},e_{22,29},e_{24,49},e_{26,39},e_{32,45},e_{34,41},e_{36,37}\right\},\\\left\{e_{0,53},e_{2,3},e_{5,6},e_{8,9},e_{13,14},e_{16,17},e_{19,20},e_{24,49},e_{26,39},e_{28,29},e_{32,45},e_{34,41},e_{36,37}\right\},\\\left\{e_{0,53},e_{2,3},e_{5,6},e_{8,9},e_{13,14},e_{16,17},e_{19,20},e_{24,49},e_{26,39},e_{29,30},e_{32,45},e_{34,41},e_{36,37}\right\},\\\left\{e_{0,53},e_{2,3},e_{5,6},e_{10,11},e_{13,14},e_{18,43},e_{21,22},e_{26,39},e_{32,45},e_{34,41},e_{36,37},e_{47,48},e_{50,51}\right\},\\\left\{e_{0,53},e_{2,3},e_{5,6},e_{10,11},e_{13,14},e_{18,43},e_{22,23},e_{26,39},e_{32,45},e_{34,41},e_{36,37},e_{47,48},e_{50,51}\right\},\\\left\{e_{0,53},e_{2,3},e_{5,6},e_{10,11},e_{13,14},e_{18,43},e_{22,29},e_{26,39},e_{32,45},e_{34,41},e_{36,37},e_{47,48},e_{50,51}\right\},\\\left\{e_{0,53},e_{2,3},e_{7,8},e_{10,11},e_{13,14},e_{16,23},e_{18,43},e_{20,33},e_{26,39},e_{28,35},e_{30,31},e_{47,48},e_{50,51}\right\},\\\left\{e_{0,53},e_{2,3},e_{7,8},e_{10,11},e_{13,14},e_{18,43},e_{20,33},e_{22,23},e_{26,39},e_{28,35},e_{30,31},e_{47,48},e_{50,51}\right\},\\\left\{e_{0,53},e_{2,3},e_{7,8},e_{10,11},e_{13,14},e_{18,43},e_{20,33},e_{23,24},e_{26,39},e_{28,35},e_{30,31},e_{47,48},e_{50,51}\right\},\\\left\{e_{0,53},e_{2,3},e_{7,36},e_{11,12},e_{14,27},e_{16,17},e_{21,22},e_{24,49},e_{30,31},e_{33,34},e_{38,51},e_{40,47},e_{42,43}\right\},\\\left\{e_{0,53},e_{2,3},e_{7,36},e_{11,12},e_{16,17},e_{21,22},e_{24,49},e_{26,27},e_{30,31},e_{33,34},e_{38,51},e_{40,47},e_{42,43}\right\},\\\left\{e_{0,53},e_{2,3},e_{7,36},e_{11,12},e_{16,17},e_{21,22},e_{24,49},e_{27,28},e_{30,31},e_{33,34},e_{38,51},e_{40,47},e_{42,43}\right\},\\\left\{e_{0,53},e_{2,15},e_{4,5},e_{9,10},e_{12,37},e_{18,19},e_{21,22},e_{24,49},e_{26,39},e_{28,35},e_{30,31},e_{41,42},e_{44,45}\right\},\\\left\{e_{1,2},e_{4,5},e_{7,8},e_{10,17},e_{12,37},e_{14,27},e_{20,33},e_{22,29},e_{24,25},e_{41,42},e_{44,45},e_{47,48},e_{50,51}\right\},\\\left\{e_{1,2},e_{4,5},e_{7,8},e_{12,37},e_{14,27},e_{16,17},e_{20,33},e_{22,29},e_{24,25},e_{41,42},e_{44,45},e_{47,48},e_{50,51}\right\},\\\left\{e_{1,2},e_{4,5},e_{7,8},e_{12,37},e_{14,27},e_{17,18},e_{20,33},e_{22,29},e_{24,25},e_{41,42},e_{44,45},e_{47,48},e_{50,51}\right\},\\\left\{e_{1,2},e_{4,5},e_{7,8},e_{12,37},e_{14,27},e_{17,18},e_{22,29},e_{24,25},e_{31,32},e_{34,35},e_{39,40},e_{46,53},e_{50,51}\right\},\\\left\{e_{1,2},e_{4,5},e_{7,8},e_{12,37},e_{14,27},e_{18,19},e_{22,29},e_{24,25},e_{31,32},e_{34,35},e_{39,40},e_{46,53},e_{50,51}\right\},\\\left\{e_{1,2},e_{4,5},e_{7,8},e_{12,37},e_{14,27},e_{18,43},e_{22,29},e_{24,25},e_{31,32},e_{34,35},e_{39,40},e_{46,53},e_{50,51}\right\},\\\left\{e_{1,2},e_{4,5},e_{7,36},e_{9,10},e_{12,13},e_{16,23},e_{18,19},e_{25,26},e_{28,29},e_{33,34},e_{38,51},e_{40,47},e_{44,45}\right\},\\\left\{e_{1,2},e_{4,5},e_{7,36},e_{9,10},e_{13,14},e_{16,23},e_{18,19},e_{25,26},e_{28,29},e_{33,34},e_{38,51},e_{40,47},e_{44,45}\right\},\\\left\{e_{1,2},e_{4,5},e_{7,36},e_{9,10},e_{13,42},e_{16,23},e_{18,19},e_{25,26},e_{28,29},e_{33,34},e_{38,51},e_{40,47},e_{44,45}\right\},\\\left\{e_{1,2},e_{4,5},e_{7,36},e_{9,10},e_{13,42},e_{16,23},e_{20,21},e_{25,26},e_{28,29},e_{31,32},e_{38,51},e_{46,53},e_{48,49}\right\},\\\left\{e_{1,2},e_{4,5},e_{7,36},e_{9,10},e_{16,23},e_{20,21},e_{25,26},e_{28,29},e_{31,32},e_{38,51},e_{41,42},e_{46,53},e_{48,49}\right\},\\\left\{e_{1,2},e_{4,5},e_{7,36},e_{9,10},e_{16,23},e_{20,21},e_{25,26},e_{28,29},e_{31,32},e_{38,51},e_{42,43},e_{46,53},e_{48,49}\right\},\\\left\{e_{1,2},e_{6,31},e_{8,21},e_{11,12},e_{16,23},e_{18,19},e_{25,26},e_{28,29},e_{33,34},e_{40,47},e_{44,45},e_{49,50},e_{52,53}\right\},\\\left\{e_{1,2},e_{6,31},e_{8,21},e_{12,13},e_{16,23},e_{18,19},e_{25,26},e_{28,29},e_{33,34},e_{40,47},e_{44,45},e_{49,50},e_{52,53}\right\},\\\left\{e_{1,2},e_{6,31},e_{8,21},e_{12,37},e_{16,23},e_{18,19},e_{25,26},e_{28,29},e_{33,34},e_{40,47},e_{44,45},e_{49,50},e_{52,53}\right\},\\\left\{e_{1,2},e_{6,31},e_{9,10},e_{14,27},e_{20,33},e_{22,29},e_{24,25},e_{35,36},e_{38,39},e_{41,42},e_{44,45},e_{47,48},e_{52,53}\right\},\\\left\{e_{1,2},e_{6,31},e_{10,11},e_{14,27},e_{20,33},e_{22,29},e_{24,25},e_{35,36},e_{38,39},e_{41,42},e_{44,45},e_{47,48},e_{52,53}\right\},\\\left\{e_{1,2},e_{6,31},e_{10,17},e_{14,27},e_{20,33},e_{22,29},e_{24,25},e_{35,36},e_{38,39},e_{41,42},e_{44,45},e_{47,48},e_{52,53}\right\},\\\left\{e_{1,30},e_{3,4},e_{7,8},e_{12,13},e_{15,16},e_{18,43},e_{24,25},e_{27,28},e_{32,45},e_{34,41},e_{38,39},e_{47,48},e_{52,53}\right\},\\\left\{e_{1,30},e_{3,4},e_{7,36},e_{10,17},e_{14,15},e_{19,20},e_{22,23},e_{25,26},e_{32,45},e_{40,47},e_{42,43},e_{49,50},e_{52,53}\right\}\},\\
\left\{e_{1,30},e_{3,4},e_{8,9},e_{12,13},e_{15,16},e_{18,43},e_{24,25},e_{27,28},e_{32,45},e_{34,41},e_{38,39},e_{47,48},e_{52,53}\right\},\\
\left\{e_{1,30},e_{3,4},e_{8,21},e_{12,13},e_{15,16},e_{18,43},e_{24,25},e_{27,28},e_{32,45},e_{34,41},e_{38,39},e_{47,48},e_{52,53}\right\},\\
\left\{e_{1,30},e_{3,4},e_{10,17},e_{14,15},e_{19,20},e_{22,23},e_{25,26},e_{32,45},e_{35,36},e_{40,47},e_{42,43},e_{49,50},e_{52,53}\right\},\\
\left\{e_{1,30},e_{3,4},e_{10,17},e_{14,15},e_{19,20},e_{22,23},e_{25,26},e_{32,45},e_{36,37},e_{40,47},e_{42,43},e_{49,50},e_{52,53}\right\}\\
\left\{e_{1,30},e_{3,44},e_{5,6},e_{8,9},e_{11,12},e_{14,15},e_{17,18},e_{22,23},e_{25,26},e_{33,34},e_{38,51},e_{46,53},e_{48,49}\right\}.$

	\bibliographystyle{alpha}
	\bibliography{linkconditionbib}
	{\sc{DPMMS, Centre for Mathematical Sciences, Wilberforce Road, Cambridge, CB3 0WB, UK}}\\
	\emph{E-mail address}: cja59@cam.ac.uk
\end{document}